\documentclass[11pt,leqno]{amsart}
\usepackage{latexsym, amsmath, amssymb, esint}
\usepackage{bm}
\usepackage{graphics,epsfig,graphicx,graphics}
\usepackage{color, xcolor}
\usepackage{tikz}
\usepackage[hyperpageref]{backref}
\usetikzlibrary{patterns}
\usepackage[english]{babel}
\usepackage[colorlinks=true,urlcolor=blue, citecolor=red,linkcolor=blue,
linktocpage,pdfpagelabels, bookmarksnumbered,bookmarksopen]{hyperref}

\newtheorem{theorem}{Theorem}[section]

\newtheorem{lemma}[theorem]{Lemma}
\newtheorem{proposition}[theorem]{Proposition}

\theoremstyle{definition}
\newtheorem{remark}[theorem]{Remark}

\newcommand{\R}{{\mathbb R \,}}

\newcommand{\QQ}{\mathcal{Q}}

\newcommand{\BHO}{B_{H_0}}
\newcommand{\ep}{\varepsilon}

\newcommand{\pnud}{H(\nabla u_\delta) \nabla H(\nabla u_\delta) \cdot \nu}
\newcommand{\pnuz}{H(\nabla u_0) \nabla H(\nabla u_0) \cdot \nu}

\numberwithin{equation}{section}

\begin{document}
    \title[Gradient blow-up in anisotropic perfect conductivity problems]{Gradient estimates for the perfect conductivity problem in anisotropic media}

  \date{}
%
%
%
%
%
\author{Giulio Ciraolo and Angela Sciammetta}

 \address{Giulio Ciraolo \\ Dipartimento di Matematica e Informatica \\ Universit\`a di Palermo\\ Via Archirafi 34\\ 90123 Palermo\\ Italy}
\email{giulio.ciraolo@unipa.it}

 \address{Angela Sciammetta \\ Dipartimento di Matematica e Informatica \\ Universit\`a di Palermo\\ Via Archirafi 34\\ 90123 Palermo\\ Italy}
\email{angela.sciammetta@unipa.it}



\begin{abstract}
We study the perfect conductivity problem when two perfectly conducting inclusions are closely located to each other in an anisotropic background medium. We establish optimal upper and lower gradient bounds for the solution in any dimension which characterize the singular behavior of the electric field as the distance between the inclusions goes to zero.
\end{abstract}

\maketitle


\section{Introduction}
When two perfectly conducting inclusions are located closely to each other, the electric field may become arbitrarily large as the distance between the inclusions goes to zero. We aim at establishing optimal estimates for the electric field as the distance between the inclusions goes to zero. The background medium may be anisotropic, with anisotropy determined by a norm in $\R^N$, $N \geq 2$.

\subsection{Gradient estimates for the conductivity problem}

Let $\Omega \subset \R^N$, $N \geq 2$, be a domain representing the background medium. Denoting the two inclusions by $D^1_\delta,D^2_\delta \subset \Omega$,
where
$
\delta = {\rm dist} (D_\delta^1,D_\delta^2)
$
is assumed to be \emph{small},
the perfectly conductivity problem is formulated as follows
\begin{equation} \label{perf_cond_isotr}
\begin{cases}
\Delta u = 0 & \text{ in } \Omega_\delta \\
|\nabla u|=0 & \text{ in } D^i_\delta \,, i=1,2, \\
\displaystyle\int_{\partial D^i_\delta} u_\nu = 0 & i=1,2,  \\
u = \varphi & \text{ on } \partial \Omega \,,
\end{cases}
\end{equation}
where $\varphi \in C^0(\partial \Omega)$ is some given potential prescribed on the boundary of $\Omega$.

Problem \eqref{perf_cond_isotr} may be regarded as a conductivity problem in the context of electromagnetism
or as an anti-plane elasticity problem in the context of elasticity, and the gradient of the solution $u$ is either the electrical field or the stress, respectively. Furthermore, problem \eqref{perf_cond_isotr} may be seen as a limit case (for $k \to +\infty$) of the classical conductivity problem
\begin{equation} \label{CP}
\begin{cases}
\text{div}\left(a_k(x)\nabla u \right) = 0& \hbox{in } \,  \Omega, \\
u=\varphi & \hbox{on } \, \partial \Omega,
\end{cases}
\end{equation}
where
\begin{equation*}
a_k(x)=\left\{
  \begin{array}{ll}
    1, & \hbox{$\Omega$,} \\
    k, & \hbox{$D^1_\delta \cup D^2_\delta$,}
  \end{array}
\right.
\end{equation*}
with $k \in (0,+\infty)$ (see for instance \cite{BaoLiYin}).

Assuming that $D^1_\delta$ and $D^2_\delta$ are smooth and far away from the boundary of $\Omega$,
the problem of estimating $|\nabla u|$ as $\delta$ goes to zero was first raised in \cite{BabuskaEtc} in relation to stress
analysis of composites and many results have been obtained in the last two decades.

Regarding the classical conductivity problem \eqref{CP} (so that $k>0$ is finite),
in \cite{BabuskaEtc} the authors observed numerically that $\|\nabla u_\delta\|_{L^\infty(\Omega)}$ is bounded independently of the distance $\delta$ between $D^1_\delta$ and $D^2_\delta$. This result was proved rigorously by Bonnetier and Vogelius \cite{BonnVog} for $N=2$ and assuming $D^1$ and $D^2$ to be two unit balls, and it was extended by Li and Vogelius in \cite{LiVog} to general second order elliptic equations with piecewise smooth coefficients (see also \cite{LiNir} where Li and Nirenberg considered general second order elliptic systems).

When $k$ degenerates ($k \to 0$ or $k \to +\infty$) the scenario is very different: the gradient of the solution may be unbounded as $\delta \to 0$ and the blow-up rate depends on the dimension. Indeed,  it has been proved that the optimal blow-up rate of $|\nabla u|$ is $\delta^{-1/2}$ for $N=2$, it is $(\delta |\log \delta|)^{-1}$ for $N=3$  and $\delta^{-1}$ for $N \geq 4$, see \cite{ACKLY,AKL,Ref3,BaoLiLi,BaoLiYin,BaoLiYin_II,BaoLiYin_III,BaoLiYin_IV,BaoLiYin_V,Gorb,GorbNovikov,KLY,Kangy_Yu2017,Kang_Yun2017,Novikov,LL,Yun,YunII} and references therein.

\subsection{The anisotropic conductivity problem}
Our goal is to obtain gradient estimates for the perfectly conductivity problem when the background medium is anisotropic, with anisotropy described by a norm $H$. More precisely, the involved anisotropy arises from replacing the Euclidean norm of the gradient with an arbitrary norm in the associated variational integrals. 

The kind of anisotropy considered in this paper has been widely studied in the field of anisotropic geometric functionals in the mathematical theory of crystals and composites which goes back to Wulff \cite{Wulff}. Indeed, variational problems in anisotropic media naturally arise in the study of crystals and whenever the microscopic environment of the interface of a medium is different from the one in the bulk of the substance so that anisotropic surface energies have to be considered. Moreover, these kinds of anisotropy are of strong interest in elasticity, noise-removal procedures in digital image processing, crystalline mean curvature flows and crystalline fracture theory. The literature is very wide and we just mention \cite{BNP,BP,Benviste2006,BC,BCS,Ch,CFV,CFV2,DellaPietraGavitone,KangMiltonARMA2008,NP,OBGXY,Taylor,TCH} and references therein for an interested reader.

In order to properly state the problem, it is convenient to look at problem \eqref{CP} from a variational point of view. More precisely, problem \eqref{CP} can be seen as the Euler-Lagrange equation of the variational problem\begin{equation*}
  \min_{v \in W^{1,2}_{\varphi}(\Omega)}I[v]\,,
\end{equation*}
where
\begin{equation*}
I[v]=\dfrac{1}{2}\int_{\Omega}a_k(x)|\nabla v|^2 \,,
\end{equation*}
and
\begin{equation*}
  W^{1,2}_{\varphi}(\Omega)=\left\{v \in W^{1,2}(\Omega) : v =\varphi \,\, \text{on} \,\, \partial \Omega \right\} \,.
\end{equation*}
It is well-known that there exists a unique solution $u \in W^{1,2}(\Omega)$ to \eqref{CP}, which is also the minimizer of $I$ on $ W^{1,2}_{\varphi}(\Omega)$ (see for instance \cite{BaoLiLi}).
\begin{center}
\begin{figure}[h]
\includegraphics[scale=0.5]{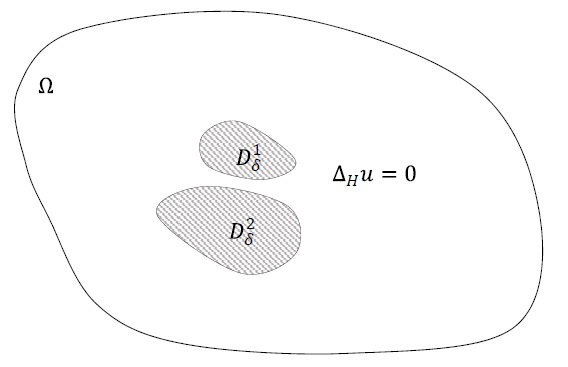}
\caption{Two perfectly conducing inclusions $D_\delta^1$ and $D_\delta^2$ are immersed in anisotropic matrix $\Omega$.}\label{FigIniziale}
\end{figure}
\end{center}

Analogously, the extreme conductivity problem \eqref{perf_cond_isotr} can be seen as the Euler-Lagrange equation of the variational problem
\begin{equation*}
\min_{v \in W^{1,2}_{\varphi}(\Omega)} \left\{ \frac12  \int_\Omega |\nabla v|^2 dx \, :  \ |\nabla v|=0 \ \text{ in } D^i_\delta \,, \ i=1,2  \right\} \,.
\end{equation*}

When the background medium is anisotropic (see Fig.\ref{FigIniziale}) the corresponding variational problem is given by
\begin{equation}\label{energia}
\min_{v \in W^{1,2}_{\varphi}(\Omega)} \left\{ \frac12  \int_\Omega H(\nabla v)^2 dx \, :  \ H(\nabla v)=0 \ \text{ in } D^i_\delta \,, \ i=1,2  \right\} \,,
\end{equation}
where $H$ is a norm in $\R^N$, $N \geq 2$; moreover, we shall assume that $H^2$ is strictly convex and of class $C^3(\R^N \setminus \{O\})$.
Since $H^2$ is a convex function with quadratic growth, problem \eqref{energia} has a solution for every bounded open set $\Omega$ and, since $H^2$ is strictly convex and sufficiently smooth, the solution is unique. Moreover (see Appendix \ref{appendix1}), the Euler-Lagrange equation associated to \eqref{energia} is 
\begin{equation} \label{DH}
\left\{
   \begin{array}{ll}
    \triangle_H u_{\delta} = 0& \hbox{in} \,\,\, \Omega_{\delta}, \\
 H(\nabla u_{\delta}) = 0 & \hbox{in} \,\,\, \overline{D}_{\delta}^i, \ i=1,2\,, \\ 
     \displaystyle\int_{\partial D_\delta^i} H\left(\nabla u_{\delta}\right)\nabla_{\xi}H\left(\nabla u_{\delta}\right)\cdot \nu ds=0 & i=1,2,\\
       u_{\delta}=\varphi(x) & \hbox{on} \,\,\, \partial \Omega \,,
    \end{array}
\right.
\end{equation}
where $\Omega_\delta = \Omega \setminus (\overline{D^1_\delta \cup D^2_\delta})$, $\nu$ is the outward normal to $\partial D_\delta^i$, and $\Delta_H$ denotes the Finsler Laplacian
$$
\Delta_H u_\delta = \text{div}\big(H\left(\nabla u_{\delta}\right) \nabla_{\xi} H\left(\nabla u_{\delta}\right) \big) \,,
$$
which has to be understood in the weak sense
$$
\int_{\Omega_\delta} H\left(\nabla u_{\delta}\right) \nabla_{\xi} H\left(\nabla u_{\delta}\right)  \cdot \nabla \phi \, dx = 0 \qquad \text{ for any } \phi \in C_0^1(\Omega_\delta) \,.
$$
Here and in the following, for $\nabla_\xi H (\nabla u)$ we mean the gradient of $H$ evaluated at $\nabla u(x)$, for $x \in \overline{\Omega}$. To avoid a confusing notation, we will use the variable $x$ for a point in the ambient space $\Omega \subset \mathbb{R}^N$, and the variable $\xi$ for a vector in the dual space (which is the ambient space of $\nabla u$).

Coming back to problem \eqref{DH}, we notice that $u_{\delta}$ is constant on each particle $D_{\delta}^i$ with $i=1,2$, i.e.
\begin{equation}\label{potenziale}
  u_{\delta}=\mathcal{U}^i_{\delta} \quad  \text{ on } D^i_{\delta} \,,
\end{equation}
with $\mathcal{U}^i_{\delta} \in \mathbb{R}$, $i=1,2$. We emphasize that $\mathcal{U}^1_{\delta}$ and $\mathcal{U}^2_{\delta}$ may be  different, and their values are unknown and are determined by solving the minimization problem \eqref{energia}.

When $\delta=0$, the corresponding perfectly conductivity problem is given by
\begin{equation} \label{D0}
\left\{
   \begin{array}{ll}
    \triangle_H u_{0} = 0& \hbox{in} \,\,\, \Omega_{0}, \\
 H(\nabla u_{0}) = 0 & \hbox{in} \,\,\, \overline{D_{0}^i}, \,\, i=1,2, \\
     \displaystyle \sum_{i=1,2} \int_{\partial D_0^i} H\left(\nabla u_{0}\right)\nabla_{\xi}H\left(\nabla u_{0}\right)\cdot \nu ds=0  \,,& \\
       u_{0}=\varphi(x) & \hbox{on} \,\,\, \partial \Omega \,.
    \end{array}
\right.
\end{equation}
We notice that the third condition in \eqref{D0} is different from third condition in \eqref{DH}, since in \eqref{D0} it is required that the sum of the two integrals on $\partial D_0^1$ and $\partial D_0^2$ vanishes. It is important to emphasize that the solution $u_0$ of \eqref{D0} is not the limit of $u_\delta$ as $\delta \to 0^+$. Even if there is some connection between $u_\delta$ and $u_0$ (see discussion below on the parameter $\mathcal{R}_0$), the behaviour of $u_\delta$ and $u_0$ is very different close to the limit touching point between the two inclusions. As we will show, $H(\nabla u_0)$ is bounded in $\Omega_0$, while $H(\nabla u_\delta)$ may have a blow-up at the limit for $\delta\to 0^+$. Understanding this phenomenon is the main goal of this paper, and the blow-up of $\nabla u_\delta$ will be characterized be the following quantity
\begin{equation} \label{R0}
\mathcal{R}_0 = \int_{\partial D^1_0} H(\nabla u_0) \nabla H(\nabla u_0) \cdot \nu \,.
\end{equation}

\subsection{Main result}
The goal of this paper is to study the gradient blow-up for problem \eqref{DH} under suitable regularity assumptions on the norm $H$. Before describing the main results, we recall some basic facts about norms in $\R^N$ (see Section \ref{Basic notations and preliminary results} for more details).

Given a norm $H$ in $\R^N$ (which we consider centrally symmetric), we denote by $H_0$ the dual norm. We recall that the sets of the form $\{H_0(x-x_0)=\text{const}\}$ are called \emph{Wulff shapes} (or \emph{anisotropic balls}).

Let $\delta \geq 0$ and let $D_\delta^1$ and $D_\delta^2$  be two perfectly conducting inclusions with $\overline{D_\delta^1}, \, \overline{D_\delta^2} \subset \Omega$ which are at distance $\delta$ one from each other. We define
$$
\Omega_{\delta} = \Omega \setminus \left(\overline{D_\delta^1 \cup D_\delta^2} \right) \,,
$$
so that, when the inclusions touch at the limit $\delta = 0$, we write
$$
\Omega_{0} = \Omega \setminus \left(\overline{D_0^1 \cup D_0^2} \right).
$$
We assume that at the limit the two particles touch only at the origin, so that 
$$
\overline{D_0^1} \cap \overline{D_0^2}=\{O\} \,.
$$

In this paper we consider the case when $D_\delta^1$ and $D_\delta^2$ are two Wulff shapes of radius $R_1$ and $R_2$, respectively, i.e.
\begin{equation*}
  D_\delta^i=B_{H_{0}}\left(x_{\delta}^i,R_i\right),
\end{equation*}
with $x_{\delta}^i=(0, \ldots, 0, t_{\delta}^i)$, with $t_{\delta}^i \in \mathbb{R}$ such that
\begin{equation*}
  H_{0}\left(x_{\delta}^i\right)=R+\dfrac{\delta}{2}, \,\, i=1,2.
\end{equation*}
Assuming that $D_\delta^1$ and $D_\delta^2$ are two Wulff shapes simplifies the calculations and the exposition. The approach can be adapted to study inclusions with boundary of class $C^3$ which are strictly convex close to the (unique) touching point.

\begin{figure}
\includegraphics[scale=0.5]{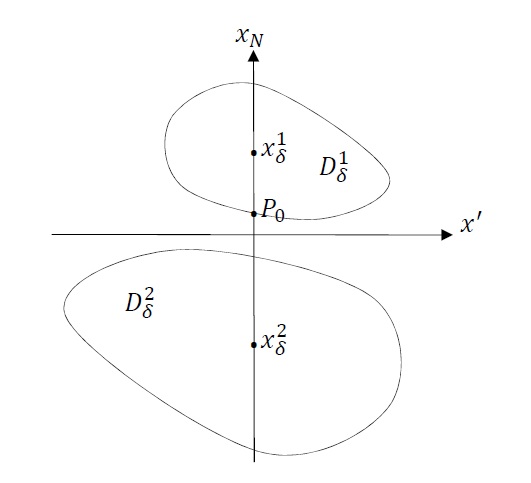}
\caption{The point $P_0$ is such that $P_0=R_1 \hat P + x_\delta^1$.} \label{fig_P0}
\end{figure}

Regarding the geometry of the problem, we recall that if $x$ is a tangency point between two Wulff shapes, then $x$ lies on the segment joining the the centers of the two sets (see Fig.\ref{fig_P0}), which is parallel to $\nabla H(\nu(x))$, where $\nu$ is the Euclidean normal (see Remark \ref{tangent} below for a proof).

We assume that
\begin{equation} \label{dist_B1B2_K}
  \text{dist}_{H_0}\left(\partial \Omega, D^1_{\delta} \cup D^2_{\delta} \right)\geq K,
\end{equation}
for some fixed $K>0$ and that the distance between the two (anisotropic) balls is very small, so that
\begin{equation*}
  \text{dist}_{H_0}\left( D^1_{\delta},  D^2_{\delta}\right) = \delta,
\end{equation*}
for some $0 < \delta \leq \delta_0$. Here, $\text{dist}_{H_0}$ denotes the distance in the ambient norm $H_0$.

Let $\hat P=(0, \ldots, 0, t_0)$ be such that $\hat P \in \partial B_{H_{0}}(0,1)$ and consider the matrix $\nabla^2H_0(\hat P)$.\footnote{Notice that, even if $\hat P$ is not univocally determined (i.e. there are two points of $\partial B_{H_{0}}(0,1)$ lying on the $x_N-$axis), the matrix $\nabla^2H_0(\hat P)$ is well defined because $H_0$ is centrally symmetric.}
We denote by $\QQ$ the matrix obtained  by considering the first $N-1$ rows and $N-1$ columns of $\nabla^2H_0(\hat P)$, i.e.
\begin{equation}\label{matriceQ}
\QQ=\left(
\begin{array}{cccc}
      \partial^2_{\xi_1\xi_1}H_0(\hat P) & \ldots & \partial^2_{\xi_1\xi_{N-1}}H_0(\hat P) \\
         \vdots & \ddots & \vdots \\
      \partial^2_{\xi_{N-1}\xi_1}H_0(\hat P) & \ldots & \partial^2_{\xi_{N-1}\xi_{N-1}}H_0(\hat P) \\
    \end{array}
  \right) \,,
\end{equation}
and recall the definition of \emph{anisotropic normal} $\nu_H$ at a point $x$, which is given by
$$
\nu_H (x) = \nabla_{\xi}H\left(\nu(x)\right) \,,
$$
where $\nu(x)$ denotes the outward Euclidean normal at $x$.
Our main result is the following.

\begin{theorem} \label{thm_main_1}
Let $u_{\delta}$ be the solution to \eqref{DH} and let $\mathcal{R}_0$ be given by \eqref{R0}. For any fixed $\tau \in (0,1/2]$ we have 
\begin{equation*}
  (1-\tau) C_*\Phi_N(\delta)+o\left(\Phi_N(\delta)\right)\leq \|\nabla u_{\delta}\|_{L^{\infty}(\Omega_{\delta})} \leq (1+\tau) C_*\Phi_N(\delta)+o\left(\Phi_N(\delta)\right)
\end{equation*}
as $\delta \to 0^+$, with
$$
\Phi_N(\delta)=\begin{cases}
\dfrac{1}{\sqrt{\delta}} & N = 2 \,, \\
\dfrac{1}{\delta|\ln\delta|} & N=3 \,, \\
\dfrac{1}{\delta} & N \geq 4 \,,
\end{cases}
$$
and
$$
C_* = \left(\dfrac{R_1+R_2}{2R_1R_2}\right)^{\frac{N-1}{2}}\left({\rm det}(\mathcal{Q})\right)^{\frac{N-1}{2}}\mathcal{R}_0 C \,,
$$
where $Q$ is given by \eqref{matriceQ} and $C$ depends on $N$ and $\nu_{H}(\hat P) \cdot \nu(\hat P)$.
\end{theorem}

We stress that the estimates in Theorem \ref{thm_main_1} are optimal, in the sense that they give the optimal rate of blow up of the gradient as $\delta \to 0$. In the Euclidean case (i.e. when $H(\cdot)=|\cdot|$) we obtain the same rate of blow up as in \cite{BaoLiYin}. We also obtain something more: the estimates in Theorem \ref{thm_main_1} almost provide a complete characterization of the leading term in the blow up. Indeed, one can choose $\tau$ arbitrarily small and get closer and closer to the sharp characterization of the blow up. The reason why we do not obtain the sharp characterization is purely technical, and how to obtain the sharp characterization is an open problem.

The strategy that we use to prove our main result has some remarkable difference compared to the one which is typically used in the Euclidean case. Indeed, in the latter case the usual approach is to use the linearity of the Laplace operator and decompose the solution $u_\delta$ in two parts:
\begin{equation} \label{decomp_linear}
u_\delta = v_\delta + w_\delta \,,
\end{equation}
where $v_\delta$ completely characterizes the asymptotic behavior of the blow-up of the gradient of $u_\delta$ and $|\nabla w_\delta|$ is uniformly bounded independently of $\delta$.

 Since $\Delta_H$ is not linear unless $H$ is an affine transformation of the Euclidean norm, we have to deal with a nonlinear problem and writing $u_\delta$ as in \eqref{decomp_linear} is not helpful. Thus we first prove that the gradient is uniformly bounded away from a \emph{small} neighborhood of the touching point and we prove, in that region, the $C^{1,\alpha}$ convergence of $u_\delta$ to $u_0$ (the solution of \eqref{D0}). Then we find estimates on the gradient in a neighborhood of the touching point and we prove optimal gradient bounds by using comparison principles and a suitable $P$-function. Our approach is purely nonlinear, and we take inspiration from \cite{GorbNovikov} where the authors study the conductivity problem in the Euclidean case for the $p$-Laplacian, with $p>N$. However, due to the presence of anisotropy and since $p \leq N$ in our case, there is some relevant difference between the two problems.
 
\medskip

The paper is organized as follows. In Section \ref{Basic notations and preliminary results} we recall some basic facts about norms in $\R^N$ and about the Finsler (or anisotropic) Laplace operator. Section \ref{sect_P_func} is devoted to prove some maximum principle, and we introduce a $P$-function which is suitable for the problem. In Section \ref{sect_unif_bounds} we prove uniform bounds on the gradient of the solution at points which are far from the touching point. Finally, in Section \ref{sect_thm1} we complete the proof of Theorem \ref{thm_main_1}.  The paper ends with two Appendixes: in the former we prove some standard facts about the perfectly conductivity problem, and in the latter we prove two techinical lemmas which are crucial for the proof of Theorem \ref{thm_main_1}.

\section{Norms and Finsler Laplacian}\label{Basic notations and preliminary results}
\emph{About norms in $\mathbb{R}^N$}. In this section we recall some facts about norms in $\mathbb{R}^N$, $N \geq 2$. Let $H:\mathbb{R}^N \rightarrow \mathbb{R}$ be a norm, i.e.
\begin{align}
& \hspace{-12em} H \,\, \text{is convex,} \label{condizione1}   \\
& \hspace{-12em} H(\xi)\geq 0 \,\, \text{for} \,\,\xi \in \mathbb{R}^N \,\, \text{and} \,\, H(\xi)=0 \,\, \text{if and only if} \,\, \xi=0, \label{condizione2}   \\
& \hspace{-12em} H(t\xi)=|t|H(\xi) \,\, \text{for} \,\, \xi \in \mathbb{R}^N \,\, \text{and} \,\, t \in \mathbb{R}.\label{condizione3} 
\end{align}

Since all norms in $\mathbb{R}^N$ are equivalent, there exist two positive constants $c_1,c_2$ such that
$$
c_1 |\xi| \leq H(\xi) \leq c_2 |\xi| \quad \text{ for any } \xi \in \mathbb{R}^N \,.
$$
The dual norm of $H$, which we denote by $H_0$, is defined by
\begin{equation}\label{normaH0}
  H_0(x)=\sup_{\xi \neq 0} \dfrac{x \cdot \xi}{H(\xi)} \quad \text{for} \quad x\in \mathbb{R}^N \,;
\end{equation}
 analogously, one can define $H$ as the dual norm of $H_0$, i.e.
\begin{equation}\label{normaH}
  H(\xi)=\sup_{x \neq 0} \dfrac{x \cdot \xi}{H_0(\xi)} \quad \text{for} \quad x \in \mathbb{R}^N \,.
\end{equation}
Following our notation, $H_0$ is a norm in the ambient space and gives the norm of a point $x \in \Omega \subset \mathbb{R}^N$ and $H$ is a norm in the dual space, which is identified with $\mathbb{R}^N$.  Indeed, we notice that the gradient of a function $u:\Omega \rightarrow \mathbb{R}^N$, evaluated at $x \in \Omega$, is the element $\nabla u(x)$ of the dual space of $\mathbb{R}^N$, which associates to any vector $y \in \mathbb{R}^N$ the number $y \cdot \nabla u $. Unless otherwise stated, we will use the variable $x$ to denote a point in the ambient space $\mathbb{R}^N$ and $\xi$ for an element in the dual space. The symbols $\nabla$ and $\nabla_{\xi}$ denote the gradients with respect to the $x$ and $\xi$ variables, respectively.

Let $H \in C^1\left(\mathbb{R}^N\setminus\{0\}\right)$, from \eqref{condizione3} we have
\begin{equation} \label{nabla_H_zero_om}
  \nabla_{\xi}H(t \xi)= \text{sign}(t)  \nabla_{\xi}H(\xi), \quad \text{for} \quad \xi \neq 0 \quad \text{and} \quad t\neq 0,
\end{equation}
and\begin{equation}\label{condizionescalare}
 \nabla_{\xi}H(\xi) \cdot \xi  =H(\xi), \quad \text{for} \quad \xi \in \mathbb{R}^N,
\end{equation}
where the left hand side is taken to be $0$ when $\xi=0$. If  $H \in C^2\left(\mathbb{R}^N\setminus\{0\}\right)$, then
\begin{equation}\label{nabla2_H_omog}
  \nabla^2_{\xi}H(t\xi)=\dfrac{1}{|t|}\nabla^2_{\xi}H(\xi), \quad \text{for} \quad \xi \neq 0 \quad \text{and} \quad t\neq 0 \,,
\end{equation}
where $\nabla^2_{\xi}$ is the Hessian operator with respect to the $\xi$ variable; we also notice that
\begin{equation}\label{nabla2_H2_omog}
  \nabla^2_{\xi}H^2(t\xi)=\nabla^2_{\xi}H^2(\xi), \quad \text{for} \quad \xi \neq 0 \quad \text{and} \quad t\neq 0 \,.
\end{equation}
Hence, \eqref{condizionescalare} implies that
\begin{equation}\label{H=0}
   \partial^2_{\xi_{i}\xi_{k}}H (\xi) \xi_{i}=0,
\end{equation}
for every $k=1, \ldots,N$.

 The following properties hold provided that $H \in C^1\left(\mathbb{R}^N\setminus\{0\}\right)$ and the unitary ball $\{\xi \in \mathbb{R}^n:\ H(\xi) < 1 \}$ is strictly convex (see \cite[Lemma 3.1]{CianchiSalani}):
\begin{equation}\label{LemmaCianchiSalani1}
 H_0\left(\nabla_{\xi}H(\xi)\right)=1, \quad \text{for} \quad \xi \in \mathbb{R}^N\setminus\{0\},
\end{equation}
and
\begin{equation}\label{LemmaCianchiSalani11}
  H\left(\nabla H_0(x)\right)=1, \quad \text{for} \quad x \in \mathbb{R}^N\setminus\{0\};
\end{equation}
furthermore, the map $H \nabla_{\xi}H$ is invertible with
\begin{equation}\label{invertibile}
  H \nabla_{\xi}H=\left(H_0 \nabla H_0\right)^{-1} \,.
\end{equation}

For $\xi_0 \in \mathbb{R}^N$ and $r>0$, the ball of center $\xi_0$ and radius $r$ in the norm $H$ is denoted by
\begin{equation*}
B_H(\xi_0,r)=\{\xi \in \mathbb{R}^N :H(\xi-\xi_0)<r\};
\end{equation*}
analgously,
\begin{equation*}
  B_{H_{0}}(x_0,r)=\{x \in \mathbb{R}^N :H_0(x-x_0)<r\}
\end{equation*}
denotes the ball of center $x_0$ and radius $r$ in the norm $H_0$. A ball in the norm $H_0$ is called the \emph{Wulff shape} of $H$.

\medskip

\emph{Assumptions on $H$}. We shall consider norms such that the unitary balls are uniformly convex. More precisely, we are considering a uniformly elliptic norm of class $C^3$ outside the origin, i.e. a function $H \in C^{3}(\mathbb{R}^N \setminus \{O\})$ for which there exists $\lambda_*, \lambda^* > 0$ such that
\begin{equation} \label{norm_unif_ellip}
 \frac{\lambda_*}{|v|} \Bigg{|} \tau - \left( \tau \cdot \frac{v}{|v|} \right) \frac{v}{|v|} \Bigg{|}^2 \leq \langle \nabla^2 H (v) \tau , \tau \rangle \leq \frac{\lambda^*}{|v|} \Bigg{|} \tau - \left( \tau \cdot \frac{v}{|v|} \right) \frac{v}{|v|} \Bigg{|}^2 \,,
\end{equation}
for every $v,\tau \in \mathbb{R}^N$, $v \neq 0$. We recall that, under these hypotheses, the boundary of the Wulff shape is uniformly convex (see, for instance, \cite{Schneider} p.111).

\medskip

\emph{Finsler Laplacian}. The Finsler Laplacian (associated to $H$) of the function $u$ is given by
\begin{equation*}
\triangle_H u=\text{div}\left(H\left(\nabla u\right)\nabla_{\xi}H\left(\nabla u\right)\right).
\end{equation*}

We recall the maximum and comparison principles for the Finsler Laplacian (see \cite[Theorem 4.1]{FeroneKawohl} and \cite[Theorem 4.2]{FeroneKawohl}).

\begin{theorem}\label{MaximumPrinciple}
If $-\triangle_H u \leq 0$ in $\Omega$ and $u = g \leq M$ on $\partial \Omega$, then $u$ attains its
maximum on the boundary; that is, $u(x) \leq M$ a.e. in $\Omega$.
\end{theorem}

\begin{theorem}\label{ComparisionPrinciple}
Suppose that $-\triangle_H u \leq -\triangle_H v$ in $\Omega$ and $u \leq v$ on $\partial \Omega$. Then $u\leq v$ a.e. in $\Omega$.
\end{theorem}
Let $\BHO(r)$ and $\BHO(R)$ be two Wulff shapes centered at the origin, with $r<R$. It will be useful to have at hand the explicit solution to the problem
\begin{equation} \label{pb_v_anello}
\left\{
   \begin{array}{ll}
    \Delta_H v = 0& \hbox{in} \,\,\, \BHO(R) \setminus \overline{\BHO(r)}, \\
       v=C_r & \hbox{on} \,\,\, \partial \BHO(r),\\
v=C_R & \hbox{on} \,\,\, \partial \BHO(R),
    \end{array}
\right.
\end{equation}
which is given by
\begin{equation} \label{v_anello}
v(x)=\left\{
   \begin{array}{ll}
    (C_r-C_R)\dfrac{H_0(x)^{2-N}-R^{2-N}}{r^{2-N}-R^{2-N}}+C_R& \hbox{if} \,\,\, N\geq 3, \\
\\
      (C_r-C_R)\dfrac{\ln \left(R^{-1} H_0(x)\right)}{\ln \left(R^{-1}r \right)}+C_R& \hbox{if} \,\,\, N=2, \\
    \end{array}
\right.
\end{equation}
for any $x \in \overline{\BHO(R)} \setminus \BHO(r)$. It is readily seen that $v \in C^3( \overline{\BHO(R)} \setminus \BHO(r))$, it satisfies \eqref{pb_v_anello}, and $H(\nabla v) \neq 0$ (see also \cite[Theorem 3.1]{FeroneKawohl}). Moreover, the following bounds
\begin{equation} \label{v_grad_bound_N3}
\frac{(N-2)|C_r-C_R| }{[(r/R)^{2-N}-1]R} \leq H(\nabla v (x)) \leq  \frac{(N-2)|C_r-C_R|}{[(r/R)^{2-N}-1]r} \, \quad \text{ for } N\geq 3\,,
\end{equation}
and
\begin{equation} \label{v_grad_bound_N2}
\frac{|C_r-C_R| }{R\ln \left(Rr^{-1} \right)} \leq H(\nabla v (x)) \leq  \frac{|C_r-C_R|}{r\ln \left(Rr^{-1} \right)} \, \quad \text{ for } N=2\,,
\end{equation}
hold for any $x \in  \overline{\BHO(R)} \setminus \BHO(r)$.

\medskip

\emph{Definition of the neck}. It will be useful to introduce the following notation. For a fixed $w>0$ sufficiently small we define the \emph{neck} of width $w$ as the set
\begin{equation}\label{neck}
  \mathcal{N}_{\delta}(w)=\{x=(x',x_N) \in \Omega_{\delta}\,\, \text{such that} \,\, |\QQ^{\frac{1}{2}}x'|<w, H_{0}(x)<r \},
\end{equation}
where $\QQ^{\frac{1}{2}}$ is the square root of the matrix $\QQ$ defined in \eqref{matriceQ}. Notice that, if $w$ is small enough, $\mathcal{N}_{\delta}(w)$ is as in Fig. \ref{FiguraConNeck}.
\begin{center}
\begin{figure}[h]
\includegraphics[scale=0.4]{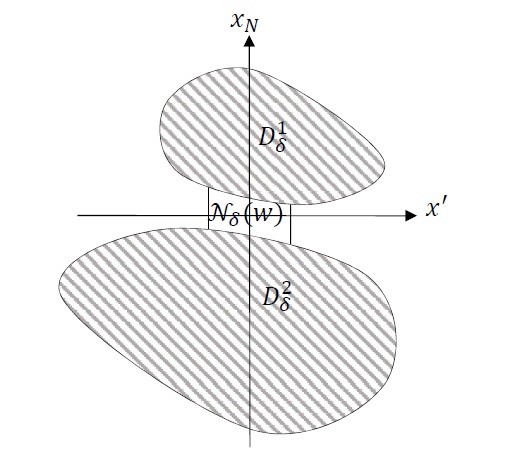}
\caption{The neck $\mathcal{N}_\delta(w)$. \label{FiguraConNeck}}
\end{figure}
\end{center}
\begin{remark}\label{tangent}
Let $B_{H_0}^1$ and $B_{H_0}^2$ be two anisotropic balls which are tangent to some point. Let $x^1 $ and $x^2$ be the centers of $B_{H_0}^1$ and $B_{H_0}^2$, respectively. Then the touching point lies on the segment joining the two centers $x^1$ and $x^2$.

Indeed, since
$$
B_{H_0}^1=\{y : H_0(y-x^1)<r_1\} \,,
$$
where $r_1$ is the radius of $B_{H_0}^1$, then
$$
\nabla H_0(x-x^1) = \gamma_1 \nu(x)\,,
$$
for some $\gamma_1>0$. We apply $H$ to both sides of the above equation and from \eqref{LemmaCianchiSalani11} we find
$$
1 = \gamma_1 H(\nu(x))\,,
$$
which yields
$$
\nabla H_0(x-x^1) = \frac{\nu(x)}{H(\nu(x))} \,.
$$
An analogous argument shows that
$$
\nabla H_0(x^2-x) = \frac{\nu(x)}{H(\nu(x))} \,.
$$
We apply $\nabla H$ in the last two equations and, by using the properties of the norms, we find
$$
\frac{x-x^1}{H_0(x-x^1)} = \nabla H(\nu(x)) = \frac{x^2-x}{H_0(x-x^2)} \,,
$$
as claimed.
\end{remark}

\section{Maximum principles} \label{sect_P_func}
In this section we prove some maximum principles for $u_\delta$, $H(\nabla u)$ and for a $P$-function which is suitable for our purposes. 

We first notice that the maximum and minimum of $u_\delta$ are attained at the boundary of $\Omega$.
\begin{lemma} \label{lemma_maxmin}
Let $u_{\delta}$ the solution of problem \eqref{DH}. The maximum and the minimum of $u_{\delta}$ are attained on $\partial \Omega$. In particular, we have that
$$
\max_{\overline{\Omega}_\delta} |u_\delta| =  \max_{\partial \Omega} |\varphi| \,.
$$
\end{lemma}
\begin{proof}
The maximum principle for the Finsler Laplacian yields that $|u_\delta|$ attains its maximum $\partial \Omega_\delta$. We show that the maximum of $u_\delta$ can not be attained at $\partial D_\delta^i$, with $i \in \{1,2\}$. Indeed, assume by contradiction that $\max u_{\delta}=T_1$. From Hopf's lemma we have that $|\nabla u_{\delta}|>0$ on $\partial D_\delta^1$, which contradicts the third condition of \eqref{DH}. Analogously, the maximum can not be attained at $\partial D_\delta^2$.
\end{proof}

Before giving other maximum principles, we set some notation and prove some basic inequalities for the Finsler Laplacian.
In order to avoid heavy formulas, we use the following notation:
$$
u_i=\dfrac{\partial u}{\partial x_i} \,,\quad  u_{ij}=\dfrac{\partial^2 u}{\partial x_i \partial x_j} \,.
$$
and
\begin{equation*}
  \partial_{\xi_i}H=\dfrac{\partial H}{\partial \xi_i}, \quad \partial^2_{\xi_i \xi_j}H=\dfrac{\partial^2 H}{\partial \xi_i \partial \xi_j}.
\end{equation*}
Since
\begin{equation*}
  \text{div} \left(H(\nabla u)\nabla_{\xi} H(\nabla u)\right) = \left( \partial_{\xi_{i}}H(\nabla u)\partial_{\xi_{j}}H(\nabla u)+ H(\nabla u)\partial_{\xi_{i} \xi_{j}}^2H(\nabla u)\right)u_{ij}
\end{equation*}
where $\nabla u \neq 0$, by setting
\begin{equation}\label{aij}
  a_{ij}:= \partial_{\xi_{i}}H(\nabla u)\partial_{\xi_{j}}H(\nabla u)+ H(\nabla u)\partial_{\xi_{i} \xi_{j}}^2H(\nabla u)= \dfrac{1}{2} \partial_{\xi_{i} \xi_{j}}^2 H(\nabla u)^2 \,.
\end{equation}
the Finsler Laplacian can be written as
\begin{equation}\label{finsleraijuj}
  \Delta_H u= a_{ij}u_{ij}= {\rm tr}(A \nabla^2 u) 
\end{equation}
at points where $\nabla u \neq 0$, where $A$ is the symmetric matrix with entries $a_{ij}$, $i,j=1,\ldots,N$. We notice that from \eqref{condizionescalare} and \eqref{H=0} we have that
\begin{equation} \label{nappa0}
a_{ij}u_i u_j = H(\nabla u)^2
\end{equation}
at points where $\nabla u \neq 0$.
It will be useful to set
\begin{equation}\label{aijk}
  a_{ijk}:=\frac12 \partial_{\xi_{i} \xi_{j} \xi_{k}}^3 H(\nabla u)^2\,,
\end{equation}
and notice that if $u$ is a solution to $\Delta_H u =0$ then
\begin{equation} \label{aijl_ul_zero}
a_{ijl} u_l = 0 ,
\end{equation}
where $\nabla u \neq 0$. Indeed, \eqref{aijl_ul_zero} can be proved by noticing that
\begin{multline*}
a_{ijl} u_l   = \partial_{\xi_{j}}H(\nabla u) \partial^2_{{\xi_{i}}{\xi_{l}}}H(\nabla u)u_l+
\partial_{\xi_{i}}H(\nabla u)\partial^2_{{\xi_{j}}{\xi_{l}}}H(\nabla u)u_l  \\
+ \partial^2_{{\xi_{i}}{\xi_{j}}}H(\nabla u) \partial_{\xi_{l}}H(\nabla u)u_l + H(\nabla u)\partial^3_{{\xi_{i}}{\xi_{j}}{\xi_{l}}} H(\nabla u) u_l \,.
\end{multline*}
By using \eqref{H=0} we obtain
\begin{equation*}
a_{ijl} u_l = H(\nabla u) \partial^2_{{\xi_{i}}{\xi_{j}}}H(\nabla u) + H(\nabla u)\partial^3_{{\xi_{i}}{\xi_{j}}{\xi_{l}}} H(\nabla u) u_l = H(\xi) \partial_{\xi_i} \left( \partial^2_{{\xi_{j}}{\xi_{l}}}H(\xi) \xi_l\right)_{|_{\xi = \nabla u}},
\end{equation*}
and from \eqref{H=0} we find \eqref{aijl_ul_zero}.

Since $\nabla_\xi^2 H^2(\xi)$ is $0$-homogeneous, the matrix $A=(a_{ij})$ satisfies
$$
|a_{ij}| \leq C_0 \quad \text{ and } \quad 2\mu_1 Id \leq A \leq 2\mu_N Id \,,
$$
where $C_0$ depends only on $\|\nabla^2 H^2\|_{C^0(\partial B_H(0,1))}$, and where $\mu_1$ and $\mu_N$ are the minimum and maximum eigenvalues of $\nabla^2_\xi H^2 $.

Let $\mathcal L$ to be the second order elliptic operator given by
\begin{equation}\label{OperatoreL}
  \mathcal{L} v := \partial_i (a_{ij} v_j) =a_{ij}v_{ij}+a_{ijl}u_{ij}v_l \,.
\end{equation}

  We recall that if $u$ is a solution to $\Delta_Hu=0$, then $u$ is of class $C^{1,\alpha}\cap W^{1,2}_{loc}$ and, by elliptic regularity, $u \in C^{2,\alpha}$ where $\nabla u \neq 0$.

  We first prove that if $u$ is a solution to $\Delta_H u=0$ then $H(\nabla u)$ satisfies a maximum principle and we also give a useful pointwise formula for $\mathcal{L}u^2$.

\begin{lemma}\label{lemma_MaxPrinc_Du_u}
Let $E \subset \R^N$ be a bounded domain and let $u$ be such that $\Delta_Hu=0$ in $E$. We have
\begin{equation} \label{P_1}
\mathcal{L} (H(\nabla u)^2) \geq \frac{2n}{n-1}\left( \partial_{\xi_i} H(\nabla u) \partial_{\xi_j} H(\nabla u)  u_{ij}\right)^2 \,,
\end{equation}
and
\begin{equation} \label{P_2}
\mathcal{L} (u^2) = 2 H(\nabla u)^2   \,.
\end{equation}
In particular, $H(\nabla u)$ satisfies the maximum principle, i.e.
\begin{equation} \label{maxprinc}
\max_{E} H(\nabla u) = \max_{\partial E} H(\nabla u)  \,.
\end{equation}
\end{lemma}
\begin{proof}
We first prove \eqref{P_2}. At points where $\nabla u \neq 0$ we have
$$
\text{div} \left(a_{ij} \nabla u^2\right) = 2a_{ij} u_i u_j + 2u a_{ij} u_{ij}+ 2u \partial_i(a_{ij}) u_j
$$
and, since $ a_{ij} u_{ij} = \Delta_H u=0$, we obtain
$$
\text{div} \left(a_{ij} \nabla u^2\right) = 2a_{ij} u_i u_j + 2u u_j a_{ijk} u_{ki}  \,.
$$
From \eqref{nappa0} and \eqref{aijl_ul_zero} we find
\begin{equation} \label{nappa2010}
\text{div} \left(a_{ij} \nabla u^2\right)  = 2 H(\nabla u)^2  \,,
\end{equation}
at points where $\nabla u \neq 0$. By continuity, \eqref{nappa2010} can be extended to zero where $\nabla u =0$.

In order to prove \eqref{P_1}, we first notice that the following Bochner formula holds (see also \cite[Lemma 2.1]{WangXia}):
\begin{equation} \label{Bochner}
a_{ij}\partial^2_{ij} H(\nabla u)^2 = 2 a_{ij}a_{kl}u_{ik}u_{jl} - \partial_{l}H(\nabla u)^2a_{ijl}u_{ij} \,,
\end{equation}
where $\nabla u \neq 0$. Indeed, \eqref{Bochner} follows from the following argument.  Owing to \eqref{aij} and \eqref{finsleraijuj} and since $\Delta_Hu=0$ we have
\begin{eqnarray*}
a_{ij} \partial_{ij}^2 \left(H(\nabla u)^2\right)&=& a_{ij} \partial_{j}\left(\partial_{\xi_{k}} H(\nabla u)^2 u_{ik}\right)=a_{ij} \underbrace{\partial^2_{\xi_{k}\xi_{l}} H(\nabla u)^2}_{2a_{kl}} u_{ik}u_{jl}+a_{ij} \partial_{\xi_{k}} H(\nabla u)^2 u_{ijk}\\
&=&2a_{ij}a_{kl}u_{ik}u_{jl}+ \partial_{\xi_{k}} H(\nabla u)^2\partial_{k}(\underbrace{a_{ij}u_{ij}}_{\Delta_H u=0})-\partial_{\xi_{k}} H(\nabla u)^2\partial_k(a_{ij})u_{ij}\\
&=&2a_{ij}a_{kl}u_{ik}u_{jl}-2H(\nabla u)\partial_{\xi_{k}} H(\nabla u)a_{ijl}u_{lk}u_{ij} \,,
\end{eqnarray*}
where $\nabla u \neq 0$, which proves \eqref{Bochner}.

  Since
\begin{equation*}
\mathcal{L} H(\nabla u)^2=a_{ij}\partial^2_{ij}H(\nabla u)^2+a_{ijl} u_{ij}\partial_l H(\nabla u)^2,
\end{equation*}
from \eqref{Bochner} we have
\begin{equation*}
\mathcal{L} H(\nabla u)^2=2a_{ij}a_{kl}u_{ik}u_{jl},
\end{equation*}
and from
\begin{equation*} 
  a_{ij}a_{kl}u_{ik}u_{jl}\geq \dfrac{(a_{ij}u_{ij})^2}{n}+\dfrac{n}{n-1}\left(\dfrac{a_{ij}u_{ij}}{n}-\partial_{\xi_i}H \partial_{\xi_j}H u_{ij}\right)^2,
\end{equation*}
(see \cite[Lemma 2.3]{WangXia}) we obtain
\begin{equation}\label{Lp2}
\mathcal{L} H(\nabla u)^2 \geq \dfrac{2n}{n-1} \left(\partial_{\xi_i} H(\nabla u) \partial_{\xi_j} H(\nabla u)  u_{ij}\right)^2 \,,
\end{equation}
at point where $\nabla u \neq 0$. We set $E_0=\{x \in E :\ \nabla u = 0\}$; since $u\in C^{1,\alpha}$ then $E_0$ is closed. From \eqref{Lp2} we have that $H(\nabla u)^2$ satisfies a maximum principle in $E \setminus E_0$ and hence $\max H(\nabla u)^2$ is attained at $\partial E \cup \partial E_0$. Since $H(\nabla u)=0$ in $ E_0$, we have that $ H(\nabla u)$ attains the maximum at $\partial E$ and \eqref{maxprinc} follows.
\end{proof}

Now we prove a maximum principle for a $P$-function which is suitable for our problem, which take care of the presence of the neck $\mathcal{N}_{\delta}(w)$, $w>0$ (see formula \eqref{neck} for its definition).

 In the following we write $x \in \R^N$ as $x=(x',x_N)$, where $x' \in \R^{N-1}$ and $x_N \in \R$. We will need to introduce a cut-off function $f \in C^2 (\overline \Omega)$ such that
\begin{equation} \label{f_1}
|f| =1 \text { in } \overline\Omega_\delta \setminus \mathcal{N}_{\delta}(w) \,, \quad f= 0 \text { in } \mathcal{N}_{\delta}\left(\frac w2 \right) \, .
\end{equation}
Moreover we choose $f$ such that
\begin{equation} \label{f_2}
 \frac{f}{w} \leq |\nabla f|^2 \quad \text{ and } \quad |\nabla^2 f | \leq \frac{1}{w^2} \,.
\end{equation}
(see \cite{GorbNovikov} for an explicit example in the Euclidean case).

\begin{theorem}\label{MaximumPrinciple}
Let $u_\delta$ be such that $\Delta_H u_\delta = 0$ in $\Omega_\delta$.
Let $f$ satisfy \eqref{f_1} and \eqref{f_2}.

There exists $\lambda_0=\lambda_0(\|f\|_{C^2} , \|H\|_{C^3(\partial B_H(0,1))})$, with $\lambda_0 = O(w^{-2})$ as $w \to 0^+$, such that the function
\begin{equation}\label{pfunction}
  P(x)= f(x) H(\nabla u)^2+ \lambda u^2
\end{equation}
satisfies the maximum principle for any $\lambda \geq \lambda_0$, i.e.
\begin{equation} \label{P_3}
\max_{x \in\overline{\Omega}_\delta} P(x) = \max_{x \in \partial \Omega_\delta} P(x)
\end{equation}
for $\lambda \geq \lambda_0$.
\end{theorem}

\begin{proof}
We first prove the assertion when the maximum is attained at a point $x_0$ where $\nabla u (x_0)= 0$, and then we consider the case when $\nabla u (x_0) \neq 0$.

\emph{Step 1}. Suppose that $P$ attains the maximum at point $x_0$ such that $\nabla u(x_0)=0$. Then $P(x_0) =  \lambda u_{\delta}(x_0)^2 $ and
$$
f(x) H(\nabla u_\delta(x))^2 + \lambda u_\delta(x)^2 \leq \lambda u_{\delta}(x_0)^2 
$$
for any $x \in \Omega_\delta$. In particular $|u_\delta(x)| \leq  |u_{\delta}(x_0)|$, and Lemma \ref{lemma_maxmin} yields that $x_0 \in \partial \Omega$.

\emph{Step 2}. Suppose that $P$ attains the maximum at a point $x_0$ such that $\nabla u(x_0)\neq0$.
From \eqref{P_1} and \eqref{P_2} we have
\begin{eqnarray}\label{LP}
  \mathcal{L}P &\geq&  a_{ij}\partial_{ij} f(x) H(\nabla u)^2 +2a_{ij}\partial_i f(x)\partial_{j}H(\nabla u)^2+ a_{ijl} \partial_l f(x)  H(\nabla u)^2 u_{ij}\\
&+ & f(x) \dfrac{2n}{n-1}\left( \partial_{\xi_i} H(\nabla u) \partial_{\xi_j} H(\nabla u)  u_{ij}\right)^2+ 2\lambda H(\nabla u)^2.
\end{eqnarray}

 Since $H$ is 1-homogeneous, the quantities $a_{ij}$, $a_{ijl}H(\nabla u)$ and $\partial_{\xi_i} H$ are 0-homogeneous. Hence there exists a contant $C_0$ depending only on $\|H\|_{C^3(\partial B_H(0,1))}$ such that
\begin{equation}\label{limitati}
|a_{ij}|, |a_{ijl}H(\nabla u)| \leq C_0 \,,\text{and} \, C_0^{-1}\leq |\partial_{\xi_i} H(\nabla u)| \leq C_0.
\end{equation}
From \eqref{LP}, \eqref{limitati} and by using Cauchy-Schwarz inequality, we have\begin{eqnarray*}
  \mathcal{L}P &\geq& \left(\lambda - C_0 \|\nabla^2 f\|_{C^0}\right)H(\nabla u)^2 + 2 \sqrt{\dfrac{2n \lambda f}{n-1}}H(\nabla u)\left|\partial_{\xi_i} H(\nabla u) \partial_{\xi_j} H(\nabla u)  u_{ij}\right|\\
  &-& 4C_0^2\|\nabla^2 f\|_{C^0}\left|\nabla^2 u\right|H(\nabla u)-C_0 \|\nabla f\|_{C^0}\left|\nabla^2 u\right|H(\nabla u)\\
   &=& \left(\lambda - C_0 \|\nabla^2 f\|_{C^0}\right)H(\nabla u)^2 + \left(2C_0^{-2} \sqrt{\dfrac{2n \lambda f}{n-1}}-
 4C_0^2\|\nabla^2 f\|_{C^0}\right)\left|\nabla^2 u\right|H(\nabla u).
\end{eqnarray*}
We can choose $\lambda_0$ large enough such that
$$\lambda - C_0 \|\nabla^2 f\|_{C^0}\geq 0,$$
and
$$2C_0^{-2} \sqrt{\dfrac{2n \lambda f}{n-1}}-4
 C_0^2\|\nabla f\|_{C^0}\geq 0,$$
for $\lambda \geq \lambda_0$. The constant $\lambda_0$ depends only on $\|H\|_{C^3(\partial B_H(0,1))}$ and $\|f\|_{C^2}$, and $\lambda_0 = O(w^{-2})$. From step 1 and step 2 we conclude.
\end{proof}

\section{Uniform bounds for the gradient} \label{sect_unif_bounds}
In this section we give estimates in the region where the gradient remains uniformly bounded.
In the next lemma we show that, since the inclusions are far away from the boundary of $\Omega$, we have that the gradient of $u_\delta$ is uniformly bounded on $\partial \Omega$ independently of $\delta$.

\begin{lemma} \label{lemma_bound_partial_Omega}
Let $u_{\delta}$ be the solution of \eqref{DH}. There exists a constant $C > 0$ independent of $\delta$ such that
\begin{equation}\label{UniformeLimitatezza}
  \max_{\partial \Omega}H(\nabla u_{\delta})\leq C.
\end{equation}
\end{lemma}

\begin{proof}
Let $A \subset \Omega$ be a smooth set such that $\{x \in \Omega:\ {\rm dist}(x, \partial \Omega) > K/2 \} \subset A $, with $K$ given by \eqref{dist_B1B2_K}, and $\overline{A} \subset \Omega$. It is clear that $D^1_{\delta}$ and $D^2_{\delta}$ are contained in $A$ for any $\delta \leq \delta_0$.

Let $v_*$ and $v^*$ be the solutions to
$$
\begin{cases}
\Delta_H v_* = 0 &  \text{ in } \Omega \setminus \overline{A} \,, \\
v_*= \varphi &  \text{ on } \partial \Omega \,, \\
v_*= \min_{\partial \Omega} \varphi & \text{ on } \partial A \,,
\end{cases}
$$
and
$$
\begin{cases}
\Delta_H v^* = 0 &  \text{ in } \Omega \setminus \overline{A} \,, \\
v^*= \varphi &  \text{ on } \partial \Omega \,, \\
v^*= \max_{\partial \Omega} \varphi & \text{ on } \partial A \,,
\end{cases}
$$
respectively. From Lemma \ref{lemma_maxmin}, it is clear that $v_*$ and $v^*$ are, respectively, a lower and an upper barrier for $u_\delta$ at any point on $\partial \Omega$. Hence, the normal derivative of $u_\delta$ can be bounded in terms of the gradient of $v_*$, $v^*$, and thus $H(u_\delta)$ can be bounded by some constant $C$ which depends only on $K$ and $\varphi$, which implies \eqref{UniformeLimitatezza}.
\end{proof}

Now we show that the gradient is uniformly bounded on the boundary of the inclusions at the points which are not in the neck.

\begin{lemma} \label{lemma_bound_partial_B}
Let $u_{\delta}$ be the solution of \eqref{DH} and let $w>0$ be fixed. There exists a constant $C > 0$ independent of $\delta$ such that
\begin{equation}\label{UniformeLimitatezza}
  \max_{\partial D^i_{\delta} \setminus \partial \mathcal{N}_{\delta}(w)} H(\nabla u_{\delta})\leq C \,, \quad \quad i=1,2.
\end{equation}
\end{lemma}

\begin{proof}
Let $z \in \partial D^1_{\delta} \setminus \partial \mathcal{N}_{\delta}(w)$ and, for $r_1 \ll R_1$, denote by $B_{H_0}(z_0,r_1)$ the interior anisotropic ball of radius $r_1$ and center $z_0$ tangent to $\partial D^1_{\delta}$ at $z$, i.e. $B_{H_0}(z_0,r_1) \subset D^1_{\delta}$ and $\partial B_{H_0}(z_0,r_1) \cap \partial D^1_{\delta} = \{z\}$ (as follows from the uniform convexity of the norm, see Fig. \ref{Figurar1r2}).
\begin{center}
\begin{figure}[h]
\includegraphics[scale=0.5]{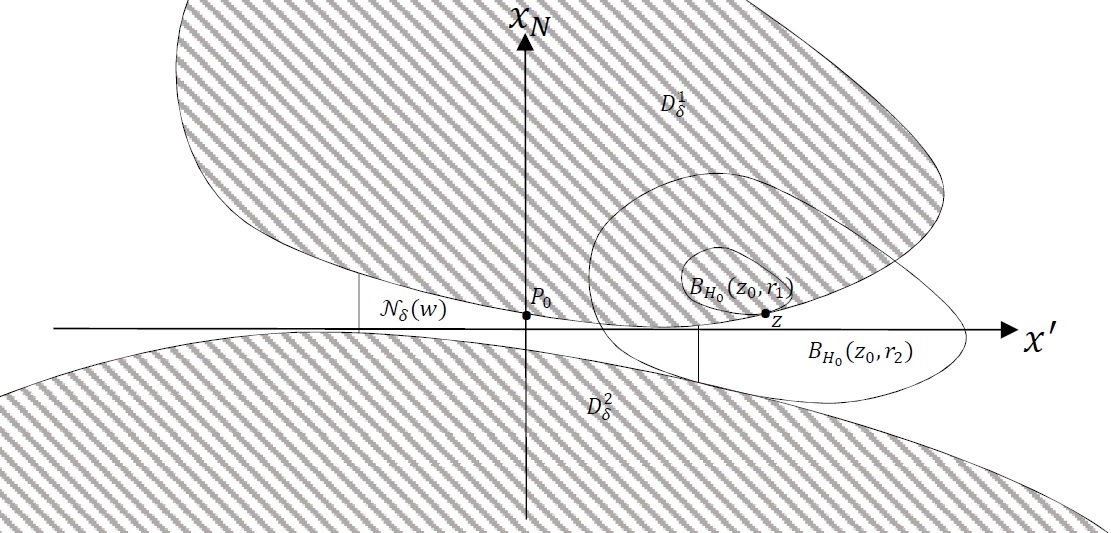}
\caption{$B_{H_0}(z_0,r_1)$ and $B_{H_0}(z_0,r_2)$ are used to construct upper and lower barriers for $u_\delta$ at $z$. \label{Figurar1r2}}
\end{figure}
\end{center}
  Let $r_2$ be the distance of $z_0$ from $\partial D^2_{\delta}$; notice that $r_2>r_1$ and the (anisotropic) ball $B_{H_0}(z_0,r_2)$ is exterior and tangent to $D^2_{\delta}$ at some point $z_1 \in \partial D^2_{\delta}$.

We construct an upper barrier $\overline{v}$ and a lower barrier $\underline{v}$ for $u_\delta$ at $z$ by considering the solutions to
\begin{equation*}
\begin{cases}
\Delta_H \overline{v} = 0 & \text{in } B_{H_0}(z_0,r_2) \setminus \overline{B}_{H_0}(z_0,r_1)\,, \\
\overline{v}=\mathcal{U}^1_{\delta} & \text{on } \partial B_{H_0}(z_0,r_1),\\
\overline{v}=\displaystyle\max_{\partial \Omega} \varphi & \text{on } \partial B_{H_0}(z_0,r_2),
\end{cases}
\end{equation*}
and
\begin{equation*}
\begin{cases}
\Delta_H \underline{v} = 0 & \text{in } B_{H_0}(z_0,r_2) \setminus \overline{B}_{H_0}(z_0,r_1)\,, \\
\underline{v}=\mathcal{U}^1_{\delta} & \text{on } \partial B_{H_0}(z_0,r_1),\\
\underline{v}=\displaystyle\min_{\partial \Omega} \varphi & \text{on } \partial B_{H_0}(z_0,r_2),
\end{cases}
\end{equation*}
respectively, when $\mathcal{U}^i_{\delta}$ are defined in \eqref{potenziale}. As follows from \eqref{v_anello} we have that
$\overline{v}$ and $\underline{v}$ are given by
\begin{equation*}
\overline{v}(x)=
\begin{cases}
\displaystyle(\mathcal{U}^1_{\delta}-\max_{\partial \Omega} \varphi)\dfrac{H_0(x-z_0)^{2-N}-r_2^{2-N}}{r_1^{2-N}-r_2^{2-N}}+\max_{\partial \Omega} \varphi & \text{if } N\geq 3, \\
 & \\
\displaystyle(\mathcal{U}^1_{\delta}-\max_{\partial \Omega} \varphi)\dfrac{\ln (r_2^{-1} H_0(x-z_0))}{\ln (r_2^{-1}r_1 )}+\max_{\partial \Omega} \varphi  & \text{if }  N=2, \end{cases}
\end{equation*}
and
\begin{equation*}
\underline{v}(x)=
\begin{cases}
\displaystyle(\mathcal{U}^1_{\delta}-\min_{\partial \Omega} \varphi)\dfrac{H_0(x-z_0)^{2-N}-r_2^{2-N}}{r_1^{2-N}-r_2^{2-N}}+\min_{\partial \Omega} \varphi & \text{if } N\geq 3, \\
 & \\
\displaystyle(\mathcal{U}^1_{\delta}-\min_{\partial \Omega} \varphi)\dfrac{\ln (r_2^{-1} H_0(x-z_0))}{\ln (r_2^{-1}r_1 )}+\min_{\partial \Omega} \varphi  & \text{if }  N=2, \end{cases}
\end{equation*}
respectively. In particular, by using \eqref{LemmaCianchiSalani11} we have
\begin{equation*}
H(\nabla\overline{v}(z))=
\begin{cases}
(N-2)\left|\mathcal{U}^1_{\delta}-\displaystyle\max_{\partial \Omega} \varphi\right|\dfrac{r_1^{1-N}}{r_1^{2-N}-r_2^{2-N}} & \text{if } N\geq 3, \\
\\
\left|\mathcal{U}^1_{\delta}-\displaystyle\max_{\partial \Omega} \varphi\right|\dfrac{r_1^{-1}}{\ln \left(r_2r_1^{-1}\right)}& \text{if } N=2,
\end{cases}
\end{equation*}
and
\begin{equation*}
H(\nabla\underline{v}(z))=
\begin{cases}
(N-2)\left|\mathcal{U}^1_{\delta}-\displaystyle\min_{\partial \Omega} \varphi\right|\dfrac{r_1^{1-N}}{r_1^{2-N}-r_2^{2-N}} & \text{if } N\geq 3, \\
\\
\left|\mathcal{U}^1_{\delta}-\displaystyle\min_{\partial \Omega} \varphi\right|\dfrac{r_1^{-1}}{\ln \left(r_2 r_1^{-1} \right)}& \text{if } N=2.
\end{cases}
\end{equation*}
We fix $r_1=c w$ for some small constant $c>0$. Since $w>0$ is fixed, there exists a constant $\alpha>1$ such that $r_2 \geq \alpha r_1$ for any $\delta \geq 0$, with $\alpha$ not depending on $\delta$. Hence we have that
$$
\dfrac{r_1^{1-N}}{r_1^{2-N}-r_2^{2-N}} \leq \frac{1}{cw(1 - \alpha^{2-N})} \quad \text{ for } N \geq 3 \,,
$$
and
$$
\dfrac{r_1^{-1}}{\ln \left(r_2r_1^{-1} \right)} \leq \frac{1}{cw \ln \alpha} \quad \text{ for } N=2 \,.
$$
Since the maximum and minimum of $u_\delta$ are attained on $\partial \Omega$ (see Lemma \ref{lemma_maxmin}) then by comparison principle we obtain that
$$
H(\nabla u_\delta (z)) \leq C
$$
where $C$ depends only on the dimension $N$, $\|\varphi\|_{C^0(\partial \Omega)}$ and $w$, and does not depends on $\delta$.
\end{proof}

\begin{lemma} \label{lemma_bound_Omega_neck}
Let $u_{\delta}$ be the solution of \eqref{DH} and let $w>0$. There exists a constant $C > 0$ independent of $\delta$ and $w$ such that
\begin{equation}\label{UniformeLimitatezza}
  \max_{\overline\Omega_\delta \setminus \Pi_w}H(\nabla u_{\delta})\leq \frac C w.
\end{equation}
\end{lemma}

\begin{proof}
From Lemma \ref{lemma_MaxPrinc_Du_u} we know that $H(\nabla u)$ satisfies the maximum principle, so that
$$
\max_{\overline\Omega_\delta \setminus \mathcal{N}_{\delta}(w)}H(\nabla u_{\delta}) \leq \max_{\partial (\Omega_\delta \setminus \mathcal{N}_{\delta}(w))}H(\nabla u_{\delta}) \,.
$$
From Lemmas \ref{lemma_bound_partial_Omega} and \ref{lemma_bound_partial_B} it is enough to find uniform bounds on $H(\nabla u)$ on $\partial \mathcal{N}^{\pm}_{\delta}(w)$, where
\begin{equation*}
  \partial \mathcal{N}^{\pm}_{\delta}(w)= \partial \mathcal{N}_{\delta}(w)\cap \{|\QQ x'|=\pm w\},
\end{equation*}
i.e. we aim at showing that there exists a constant $C$ independent on $\delta$ and $w$ such that
\begin{equation} \label{grad_bound_Pi_pm}
  \max_{\partial \mathcal{N}^{\pm}_{\delta}(w)} H\left(\nabla u_{\delta} \right)\leq \frac{C}{w}.
\end{equation}
Let $P$ be as in Theorem \ref{MaximumPrinciple} (see formula \eqref{pfunction}).

 Since $0 \leq f \leq 1$ and $f=1$ on $\mathcal{N}^{\pm}_{\delta}(w)$, we have that
$$
\max_{\partial\mathcal{N}^{\pm}_{\delta}(w)} H(\nabla u_{\delta})^2
 = \max_{\partial\mathcal{N}^{\pm}_{\delta}(w)}f(x) H(\nabla u_{\delta})^2 \leq \max_{\partial\mathcal{N}^{\pm}_{\delta}(w)} P(x) \leq \max_{\overline \Omega_\delta} P(x) \,.
 $$
From Theorem \ref{MaximumPrinciple} there exists a constant $\lambda_0 = O(w^{-2})$ such that \eqref{pfunction} satisfies the maximum principle for any $\lambda \geq \lambda_0$ and the chain of inequalities above yields
$$
\max_{\partial\mathcal{N}^{\pm}_{\delta}(w)} H(\nabla u_{\delta})^2 \leq \max_{\overline \Omega_\delta} P(x) = \max_{\partial \Omega_\delta} P(x) \,.
$$
Since $\|u_\delta\|_{C^0(\Omega_\delta)} \leq \|\varphi\|_{C^0(\partial \Omega)}$ (see Lemma \ref{lemma_maxmin}) and $\lambda_0=O(w^{-2})$ (see Theorem \ref{MaximumPrinciple}), we have that there exists a constant $C$ independent of $\delta$ and $w$ such that
\begin{equation*}
  P(x)= f(x) H(\nabla u_{\delta})^2+ \lambda u_{\delta}^2\leq f(x) H(\nabla u_{\delta})^2 + Cw^{-2},
\end{equation*}
and hence
\begin{equation*}
\max_{\partial\mathcal{N}^{\pm}_{\delta}(w)} H(\nabla u_{\delta})^2\leq \max_{\partial\Omega_{\delta}}P(\mathbf{x},y)\leq \max_{\partial\Omega_{\delta}} \left[f(x) H(\nabla u_{\delta})^2 \right]+ Cw^{-2}.
\end{equation*}
Since $f=0$ in $\mathcal{N}_{\delta}(w/2)$, from Lemmas \ref{lemma_bound_partial_Omega} and \ref{lemma_bound_partial_B} we find \eqref{grad_bound_Pi_pm} and the proof is complete.
\end{proof}

Before giving the relation between $u_{\delta}$ and $u_{0}$ (see Proposition \ref{prop_udelta_conv_u0} below), in the next Lemma we show that gradient of $u_{0}$ is bounded.

\begin{lemma} \label{lem_u0_W1infty}
Let $u_0$ be the solution to \eqref{D0}. Then $H(\nabla u_0) \leq C$.
\end{lemma}

\begin{proof}
The proof is analogous to the ones of Lemmas  \ref{lemma_bound_partial_Omega} and \ref{lemma_bound_partial_B}, and we only give a sketch. Since $H(\nabla u_0)$ attains the maximum at the boundary (see Lemma \ref{lemma_MaxPrinc_Du_u}), we have to prove that $H(\nabla u_0)$ is bounded on $\partial \Omega$ and on $\partial D^1_0 \cup \partial D^2_0$. First we recall that, in view of the third condition in \eqref{D0}, the maximum and minimum of Lemma $u_0$ are attained at $\partial \Omega$. Hence, the bound on $\partial \Omega$ can be obtained as in the proof of Lemma \ref{lemma_bound_partial_Omega}. The bound on $\partial D^1_0$ (and analogously the one on $\partial D^2_0$) can be obtained by comparison principle, more precisely by comparing $u_0$ and $v_1$ and $v_2$, where $v_i$ is the solution to $\Delta_H v_i=0$ in $\Omega \setminus \overline D^1_0$, $v_i=u_0$ on $\partial D^i_0$, $i=1,2$ $v_1=\max_{\partial \Omega} \phi$ and $v_2=\min_{\partial \Omega} \phi$ on $\partial \Omega$.
\end{proof}

We are ready to show the relation between $u_\delta$ and $u_0$.

\begin{proposition} \label{prop_udelta_conv_u0}
Let $u_{\delta}$ be the solution of \eqref{DH} and $u_0$ be the solution of \eqref{D0}.

There exists a constant $0<\alpha <1 $ not depending on $\delta$ such that
\begin{equation} \label{u_delta_to_u_0}
  \lim_{\delta\rightarrow 0}\|u_{\delta} -u_{0}\|_{C^{1,\alpha}(E)}=0,
\end{equation}
for any compact set $E \subset \Omega_0$. Moreover, for any $i = 1, 2$ and for any neck $\mathcal{N}_\delta(w)$ of (sufficiently small) width $w$ we have
\begin{equation} \label{ferragosto}
   \lim_{\delta\rightarrow 0} \displaystyle\int_{\partial D^i_{\delta}\setminus \partial \mathcal{N}_{\delta}(w)} H\left(\nabla u_{\delta}\right)\nabla_{\xi}H\left(\nabla u_{\delta}\right)\cdot \nu ds= \displaystyle\int_{\partial D^i_0\setminus \partial \mathcal{N}_{\delta}(w)} H\left(\nabla u_{0}\right)\nabla_{\xi}H\left(\nabla u_{0}\right)\cdot \nu ds \,.
\end{equation}
\end{proposition}

\begin{proof}
Thanks to Lemma \ref{lemma_bound_Omega_neck} and \cite[Theorem 2]{Dibenedetto}, for any fixed $w>0$ we have that there exists $\alpha>0$ independent of $\delta$ such that
\begin{equation} \label{unif_regul}
\|u_\delta\|_{C^{1,\alpha}(\mathcal{K})} \leq C \qquad \text{ for any compact set } \mathcal{K} \subset \Omega_\delta \setminus \overline{\mathcal{N}}_{\delta}(w) \,,
\end{equation}
where $C$ is a constant independent of $\delta$.

Let $E$ be a compact set contained in $\Omega_0$. We want to show that $u_\delta$ converges to $u_0$ in $C^{1,\alpha}(E)$. Since $E$ is fixed, there exist $w,\delta_0>0$ such that $E \subset \Omega_\delta \setminus \overline{\mathcal{N}_\delta(w)} $ for any $\delta< \delta_0$. From \eqref{unif_regul} we have that $u_\delta$ converges to some function $\bar{u}$ in $C^{1,\alpha}(E)$, which satisfies $\Delta_H \bar u = 0$ in $E$. In order to show that $\bar u$ is the solution to \eqref{D0}, i.e. $\bar u = u_0$, we only need to check that $\bar u$ satisfies the third line in \eqref{D0}, i.e. that
\begin{equation} \label{third_cond}
 \int_{\partial D_0^1} H\left(\nabla \bar u \right)\nabla_{\xi}H\left(\nabla \bar u \right)\cdot \nu ds +  \int_{\partial D_0^2} H\left(\nabla \bar u \right)\nabla_{\xi}H\left(\nabla \bar u \right)\cdot \nu ds = 0 \,.
\end{equation}
We prove \eqref{third_cond} by approximation. Let $\ep>0$ be fixed and sufficiently small, and let
$$
A_\ep = (D_0^1 \cup D_0^2) + B_{H_0}(0,\ep) \,,
$$
where $A+B = \{a+b : a \in A \,\, \text{and} \,\, b \in B\}$ is the Minkowski sum between the sets $A$ and $B$. If $\delta<\ep$ then $D_\delta^1\cup D_\delta^2 \subset A$ and we have that $u_\delta$ converges to $\bar u$ in $C^{1,\alpha}(\overline \Omega \setminus A)$. From $\Delta_H u_\delta=0$ in $A_\ep \setminus \overline{D_\delta^1\cup D_\delta^2}$ and since $H(\nabla u_\delta) = 0$ in $D_\delta^1$ and $D_\delta^2$, the divergence theorem implies that
\begin{equation*}
\int_{\partial A_\ep} H\left(\nabla  u_\delta \right)\nabla_{\xi}H\left(\nabla u_\delta \right)\cdot \nu ds  = 0 \,.
\end{equation*}
By letting $\delta$ to zero and since $u_\delta \to \bar u$ in $C^{1,\alpha}$, we obtain that
\begin{equation*}
\int_{\partial A_\ep} H\left(\nabla  \bar u \right)\nabla_{\xi}H\left(\nabla \bar u \right)\cdot \nu ds  = 0 \,.
\end{equation*}
Since $\ep>0$ is arbitrary, we obtain \eqref{third_cond} and \eqref{u_delta_to_u_0} is proved.

  Once we have that $u_\delta \to u_0$ in $C^{1,\alpha}$ on compact sets, the proof of \eqref{ferragosto} follows straightforwardly (see for instance \cite[p.736-737]{GorbNovikov}).
\end{proof}

\section{Proof of Theorem \ref{thm_main_1}}\label{sect_thm1}
\emph{Step 1: upper and lower bounds on the gradient in the neck}. Let $w>0$ be fixed. We are going to find upper and lower bounds on the gradient of the solution in the neck in terms of $\mathcal{U}_\delta^1 - \mathcal{U}_\delta^2$, which we assume to be non negative (the case $\mathcal{U}_\delta^1 - \mathcal{U}_\delta^2\leq 0$ is completely analogous). In particular, we aim at showing that for any fixed $\tau \in (0,1/2)$ there exists a constant $C$ independent on $\delta$ such that
\begin{equation} \label{step1_upper_bound}
H(\nabla u_\delta(P)) \nabla_\xi H(\nabla u_\delta (P)) \cdot \nu (P) \leq - \dfrac{\mathcal{U}^1_{\delta}-\mathcal{U}^2_{\delta}}{\delta + (1+ \tau) \frac{R_1+R_2}{2R_1 R_2}  \QQ P^\perp \cdot P^\perp } (1+ o(\delta^2 + |P-P_0|^2) ) + C\,.
\end{equation}
and
\begin{equation} \label{step1_lower_bound}
H(\nabla u_\delta(P)) \nabla_\xi H(\nabla u_\delta (P)) \cdot \nu (P) \geq - \dfrac{\mathcal{U}^1_{\delta}-\mathcal{U}^2_{\delta}}{\delta +  (1-\tau) \frac{R_1+R_2}{2R_1 R_2} \QQ P^\perp \cdot P^\perp } (1+ o(\delta^2 + |P-P_0|^2) ) - C\,.
\end{equation}
for any $P \in \partial D_\delta^1 \cap \partial  \mathcal{N}_{\delta}(w)$, where $P_0 \in \partial D_\delta^1$ lies on the segment joining the two centers of $D_\delta^1$ and $D_\delta^2$, and $P^\perp$ is the projection of $P$ on the orthogonal to $P_0$ (see Fig.\ref{Figurar1r2_1}).

We start by finding a lower bound on $ \nabla u_\delta$ at  $P \in \partial D_\delta^1 \cap \partial \mathcal{N}_{\delta}(w)$.
We consider an anisotropic ball touching $\partial D_\delta^1$ at $P$ from the inside and denote it by $B_{H_0}(y_0,r_1)$, so that $B_{H_0}(y_0,r_1) \subset D_\delta^1$ and $y_0$ and $r_1$ are the center and the radius of the ball, respectively, where we let
$$
r_1 = t R_1 \,.
$$  
We denote by $r_2$ the radius of the anisotropic ball with center at $y_0$ which touches $\partial D_\delta^2$ from the outside, i.e.
$$
r_2={\rm dist}_{H_0} (y_0,\partial D_\delta^2)=\min\{H_0(x-y_0) : x \in \partial D^2_{\delta}\}.
$$

For $x \neq y_0$, let $\underline v$ be given by
\begin{equation*}
\underline{v}(x)=
\begin{cases}
(\mathcal{U}^1_{\delta}-\mathcal{U}^2_\delta)\dfrac{H_0(x-y_0)^{2-N}-r_2^{2-N}}{r_1^{2-N}-r_2^{2-N}}+\mathcal{U}^2_\delta & \text{if } N\geq 3\,, \\
 & \\
\displaystyle(\mathcal{U}^1_{\delta}-\mathcal{U}^2_\delta)\dfrac{\ln (r_2^{-1} H_0(x-y_0))}{\ln (r_2^{-1}r_1 )}+\mathcal{U}^2_\delta  & \text{if }  N=2\,.
\end{cases}
\end{equation*}
Notice that $\Delta_H \underline v=0$ in $\mathbb{R}^N \setminus \{y_0\}$ and $\underline v = \mathcal{U}_\delta^i$ on $\partial B_{H_0}(y_0,r_i)$, $i=1,2$. We notice that we can find a constant $M$, not depending on $\delta$, such that if the ratio
\begin{equation}
\mathcal{M} = \begin{cases}\label{M}
\dfrac{\mathcal{U}^1_{\delta}-\mathcal{U}^2_\delta}{r_1^{2-N}-r_2^{2-N}} & \text{if } N\geq 3\,, \\
 & \\
\dfrac{\mathcal{U}^1_{\delta}-\mathcal{U}^2_\delta}{\ln (r_2 r_1^{-1} )}  & \text{if }  N=2\,.
\end{cases}
\end{equation}
is large enough, say $\mathcal{M}>M$, then $\underline v$ is a lower barrier for $u_\delta$.

Now assume that $\mathcal{M}>M$, so that $\underline v$ is a lower barrier for $u_\delta$. Since
\begin{equation*}
H(\nabla\underline{v}(P))=
\begin{cases}
\dfrac{(\mathcal{U}^1_{\delta}-\mathcal{U}^2_{\delta})(N-2)}{r_1^{2-N}-r_2^{2-N}}r_1^{1-N} & \text{if } N\geq 3\,, \\
 & \\
\dfrac{\mathcal{U}^1_{\delta}-\mathcal{U}^2_{\delta}}{\ln(r_2r_1^{-1})}\dfrac{1}{r_1}  & \text{if }  N=2\,.
\end{cases}
\end{equation*}
from the mean value theorem we have that there exists $\bar r \in (r_1,r_2)$ such that
\begin{equation*}
H(\nabla\underline{v}(P)) =
\dfrac{\mathcal{U}^1_{\delta}-\mathcal{U}^2_{\delta}}{r_2-r_1} \left( \frac{\bar r}{r_1} \right)^{N-1}
\end{equation*}
for any $N \geq 2$, and hence
\begin{equation} \label{nabla_v_lower_bound}
H(\nabla\underline{v}(P)) \geq
\dfrac{\mathcal{U}^1_{\delta}-\mathcal{U}^2_{\delta}}{r_2-r_1}
\end{equation}
for any $N \geq 2$.

Thanks to \eqref{nabla_v_lower_bound} we can give an upper bound on the quantity $H(\nabla u_\delta(P)) \nabla_\xi H(\nabla u_\delta (P)) \cdot \nu(P)$. Indeed, since $\underline v$ is a lower barrier for $u_\delta$ then
\begin{equation}\label{-1}
\nabla_\xi H(\nabla u_\delta (P)) \cdot \nu(P) = - \nabla_\xi H(\nu(P) ) \cdot \nu(P) = - H(\nu(P) ) = -1 \,,
\end{equation}
where the last equality holds because $P$ lies on a Wulff shape. From \eqref{nabla_v_lower_bound} we find
\begin{equation} \label{artista}
H(\nabla u_\delta(P)) \nabla_\xi H(\nabla u_\delta (P)) \cdot \nu (P) \leq - \dfrac{\mathcal{U}^1_{\delta}-\mathcal{U}^2_{\delta}}{r_2-r_1} \,.
\end{equation}
If $\mathcal{M} \leq M$, from elliptic estimates we have $H(\nabla u) \leq C$, where $C$ does not depends on $\delta$. Indeed, from the mean value theorem we have
$$
 \dfrac{\mathcal{U}^1_{\delta}-\mathcal{U}^2_{\delta}}{r_2-r_1} \leq \frac{N-1}{t^{N-1} R_1^{N-1}}  M \,. 
$$
Since $\partial D_\delta^1$ is of class $C^3$, $u_\delta$ is constant on $\partial D_\delta^1$, and the distance of $P$ from $\partial D_\delta^2$ is of size $r_2-r_1$, from interior regularity estimates we have that  $H(\nabla u) \leq C$, where $C$ does not depends on $\delta$.

Hence
\begin{equation} \label{artista2}
H(\nabla u_\delta(P)) \nabla_\xi H(\nabla u_\delta (P)) \cdot \nu (P) \leq - \dfrac{\mathcal{U}^1_{\delta}-\mathcal{U}^2_{\delta}}{r_2-r_1} +  C\,.
\end{equation}
Let $r_1=tR_1$, we have
\begin{equation}\label{r2_meno_r1_inthm}
r_2 - r_1 = \delta + (1-s)\dfrac{R_1+R_2}{2R_1R_2} \mathcal{Q}P^\perp \cdot P^\perp + o(\delta^2 + |\omega|^2)
\end{equation}
as $\delta$ and $|P-P_0|$ go to zero, and where $P^\perp$ is the projection of $P$ on the orthogonal to $P_0$.
We do not prove \eqref{r2_meno_r1_inthm} here, and we postpone its proof in the Appendix \ref{Estimates for the radii of the touching balls} (see Lemma \ref{lemma_diff_radii}).
From \eqref{artista2} and \eqref{r2_meno_r1_inthm} we obtain \eqref{step1_upper_bound}.

Now we obtain the lower bound  \eqref{step1_lower_bound}. We consider a ball $B_H(\bar y, \rho_2)$ touching $\partial D^1_\delta$ at $P$ from the outside and such that the center $\bar y$ is contained in $D^2_\delta$ and we denote by $\rho_1$ the radius of the concentric ball touching $\partial D^1_\delta$ from the inside.
For $x \neq \bar y$, let $\overline v$ be given by
\begin{equation*}
\overline{v}(x)=
\begin{cases}
-(\mathcal{U}^1_{\delta}-\mathcal{U}^2_\delta)\dfrac{H_0(x-\bar y)^{2-N}-\rho_2^{2-N}}{\rho_1^{2-N}-\rho_2^{2-N}}+\mathcal{U}^1_\delta & \text{if } N\geq 3\,, \\
 & \\
-\displaystyle(\mathcal{U}^1_{\delta}-\mathcal{U}^2_\delta)\dfrac{\ln (\rho_2^{-1} H_0(x-\bar y))}{\ln (\rho_2^{-1}\rho_1 )}+\mathcal{U}^1_\delta  & \text{if }  N=2\,.
\end{cases}
\end{equation*}
The function $\overline v$ is such that $\Delta_H \overline v=0$ in $\mathbb{R}^N \setminus \{y_0\}$,  $\overline v = \mathcal{U}_\delta^1$ on $\partial B_{H_0}(y_0,\rho_2)$, and  $\overline v = \mathcal{U}_\delta^2$ on $\partial B_{H_0}(y_0,\rho_1)$.

If the ratio $\mathcal{M}$ (defined as in \eqref{M} is large enough, say $\mathcal{M}>M$ for some constant $M$ not depending on $\delta$, then $\overline v$ is an upper barrier for $u_\delta$ and we obtain that
$$
H(\nabla u_\delta (P)) \leq H(\nabla \overline{v}_\delta (P)) \leq \dfrac{\mathcal{U}^1_{\delta}-\mathcal{U}^2_{\delta}}{\rho_2-\rho_1} \left( \frac{\bar \rho}{\rho_2} \right)^{N-1}
$$
for some $\bar \rho \in (\rho_1,\rho_2)$, and we obtain
$$
H(\nabla u_\delta (P)) \leq\dfrac{\mathcal{U}^1_{\delta}-\mathcal{U}^2_{\delta}}{\rho_2-\rho_1} \,.
$$

By arguing as for the upper bound before, if $\mathcal{M}\leq M$ then we can find a constant $C$ such that $H(\nabla u_\delta (P))\leq C$. Hence, we have that
$$
H(\nabla u_\delta (P)) \leq\dfrac{\mathcal{U}^1_{\delta}-\mathcal{U}^2_{\delta}}{\rho_2-\rho_1} + C \,.
$$
By arguing as in Lemma \ref{lemma_diff_radii_2} below, we can prove that for any fixed $s \in (0,1/2)$ we have that
\begin{equation}\label{r2_meno_r1_inproof}
\rho_2 - \rho_1 = \delta + (1+s) \dfrac{R_1+R_2}{2R_1R_2} \mathcal{Q} P^\perp \cdot P^\perp + o(\delta^2 + |P-P_0|^2)
\end{equation}
as $\delta$ and $|P-P_0|$ go to zero, and where $P^\perp$ is the projection of $P$ on the orthogonal to $P_0$ and from \eqref{-1} we obtain \eqref{step1_lower_bound}.

\medskip

\emph{Step 2: Bounds on $\mathcal{U}_\delta^1 - \mathcal{U}_\delta^2.$}
In this step we aim at proving that for any fixed $\tau \in (0,1/2)$ we have that
\begin{eqnarray} \label{diff_U1U2}
&&(1-\tau) \left(\dfrac{R_1+R_2}{2R_1R_2}\right)^{\frac{N-1}{2}}\left({\rm det}(\mathcal{Q})\right)^{\frac{N-1}{2}} C |\partial_{\xi_N} H_0(P_0)| \mathcal{R}_0 \Psi_N(\delta)  + o(\Psi_N(\delta))\\
&&\leq \mathcal{U}_\delta^1 - \mathcal{U}_\delta^2 \leq (1 + \tau) \left(\dfrac{R_1+R_2}{2R_1R_2}\right)^{\frac{N-1}{2}}\left({\rm det}(\mathcal{Q})\right)^{\frac{N-1}{2}} C |\partial_{\xi_N} H_0(P_0)| \mathcal{R}_0 \Psi_N(\delta) + o(\Psi_N(\delta))\nonumber
\end{eqnarray}
where $C$ depends only on the dimension $N$ and
with 
\begin{equation} \label{PsiNdelta}
\Psi_N(\delta)=
\begin{cases}
\delta^{1/2} & N = 2 \,, \\
(\log(1/\delta))^{-1} & N=3 \,, \\
1 & N \geq 4 \,.
\end{cases}
\end{equation}
Let $w>0$ be fixed. From \eqref{DH} and the divergence theorem we have that
\begin{multline} \label{843}
0 = \int_{\partial D_\delta^1} \pnud \\ = \underbrace{\int_{\partial D_\delta^1 \cap \partial  \mathcal{N}_{\delta}(w)}  \pnud  }_{I_1} + \int_{\partial D_\delta^1 \setminus \partial  \mathcal{N}_{\delta}(w)} \pnud \,.
\end{multline}
We consider the set $E=D_0^1 \cup E_0$, where $E_0$ is some smooth fixed set containing $D_\delta^1$ and not containing $D_\delta^2$, and such that  $\partial\left(E \cap \mathcal{N}_{\delta}(w) \right)\subset \partial D_0^1$ for $w$ small enough.
  Notice that $\partial E \cap \mathcal{N}_{\delta}(w) = \partial E_1 \subset \partial D_0^1$ and
$\partial E \setminus \mathcal{N}_{\delta}(w)= \partial E_2 \subset \partial E_0 $.

  Since $\Delta_H u_\delta =0$ in $\Omega_\delta$ we apply the divergence theorem in $ B_{H_0}(x_0,R)\setminus \overline{\mathcal{N}}_{\delta}(w)$ and we have that
\begin{multline}
\int_{\partial D_\delta^1 \setminus \mathcal{N}_{\delta}(w)} \pnud =  \underbrace{\int_{\partial E_2} \pnud }_{I_2}
\\ +  \underbrace{\int_{E \cap \left(\partial\mathcal{N}^{+}_{\delta}(w))\cup \partial\mathcal{N}^{-}_{\delta}(w))\right)} \pnud}_{I_3} \,.
\end{multline}
Proposition \ref{prop_udelta_conv_u0} and Lemma \ref{UniformeLimitatezza} yield
\begin{equation} \label{adam}
 I_2 = \int_{\partial E_2} \pnuz +o(1),
\end{equation}
as $\delta \rightarrow 0$.
We recall that by definition
$$
\mathcal{R}_0 :=   \int_{\partial D_0^1} \pnuz \,.
$$
Since $u_0 \in W^{1,\infty}$ (see Lemma \ref{lem_u0_W1infty}), by applying the divergence theorem in the set
$E \setminus\overline{D}_0^1$ we have that
$$
\mathcal{R}_0= \int_{\partial E} \pnuz \,,
$$
and from \eqref{adam} we obtain
\begin{equation} \label{844}
\Big{|} I_2 + \mathcal{R}_0  \Big{|} \leq  C w^{N-1}+ o(1),
\end{equation}
as $\delta \rightarrow 0$, where $C$ does not depends on $w$. Notice that, from Lemma \ref{UniformeLimitatezza} we have
$$
\Big{|} I_3 \Big{|} \leq \frac{C \delta}{w},
$$
where $C$ does not depends on $w$.
This last estimate together with \eqref{843} and \eqref{844} yield
\begin{equation*}
| I_1 - \mathcal{R}_0 | \leq Cw^{N-1} +  \frac{C \delta}{w} + o(1) \,.
\end{equation*}
By choosing $w=\delta^{1/2}$ we have that
\begin{equation}\label{I1R0}
| I_1 - \mathcal{R}_0 | = o(1) ,
\end{equation}
as $\delta \to 0^+$.

Now we estimate $I_1$. Together with \eqref{I1R0}, this will imply upper and lower bounds on $\mathcal{U}_\delta^1-\mathcal{U}_\delta^2$. We recall that $I_1$ is given by
$$
I_1 = \int_{\mathcal{I}}  \pnud \,,
$$
where we set $\mathcal{I} = \partial D_\delta^1 \cap \partial  \mathcal{N}_{\delta}(w)$ to lighten the notation.
From \eqref{step1_upper_bound} and \eqref{step1_lower_bound}, we obtain that for ant $\tau \in (0,1/2)$ we have
\begin{equation} \label{int_1}
- \int\limits_{\mathcal{I}} \dfrac{\mathcal{U}^1_{\delta}-\mathcal{U}^2_{\delta}}{\delta +  (1-\tau) \dfrac{R_1+R_2}{2R_1R_2} \mathcal{Q} P^\perp \cdot P^\perp } d\sigma (1+ o(1)) - Cw
 \leq I_1
\end{equation}
and
\begin{equation} \label{int_2}
I_1  \leq - \int\limits_{\mathcal{I}} \dfrac{\mathcal{U}^1_{\delta}-\mathcal{U}^2_{\delta}}{\delta +  (1+ \tau)\dfrac{R_1+R_2}{2R_1R_2} \mathcal{Q} P^\perp \cdot P^\perp } d\sigma (1+ o(1)) + Cw \,.
\end{equation}
Hence, we have to understand the asymptotic behaviour of
\begin{equation} \label{integralone1}
\hat I = \int\limits_{\mathcal{I}} \dfrac{d\sigma}{\delta + c \mathcal{Q}P^\perp \cdot P^\perp }
\end{equation}
as $\delta \to 0$, where $c= (1\pm \tau)\dfrac{R_1+R_2}{2R_1R_2}>0$. Once we have that, being $I_1$ finite, the asymptotic behaviour of $\mathcal{U}^1_{\delta}-\mathcal{U}^2_{\delta}$ follows from \eqref{int_1} and \eqref{int_2}.

We notice that $P^\perp$ lies on $\{x_N = 0\}$, and so we write $P^\perp =x'=(x_1,\ldots,x_{N-1})$ for $P \in \mathcal I$. From the implicit function theorem, there exists a function $\phi : \{|\QQ^{1/2} x'|<w\} \to \mathbb{R}$ such that $H_0(x',\phi(x'))=R_1$, $\phi(0) = \delta$ and $(x',\phi(x')) \in \mathcal I$. Hence \eqref{integralone1} becomes
$$
\hat I =  \int\limits_{\{|\QQ^{1/2} x'|< w\}} \dfrac{ \sqrt{1+|\nabla_{x'} \phi(x')|^2} dx'}{\delta +  c \mathcal{Q} x' \cdot x' } \,.
$$
Since
$$
1+|\nabla_{x'} \phi(x')|^2 = \frac{|\nabla H_0(x',\phi(x'))|^2}{(\partial_{\xi_N} H_0(x',\phi(x')))^2} \,,
$$
and $(x',\phi(x'))$ lies on a Wulff shape, we find that
$$
1+|\nabla_{x'} \phi(x')|^2 = \frac{1}{(\partial_{\xi_N} H_0(x',\phi(x')))^2} =  \frac{1}{(\partial_{\xi_N} H_0(P_0))^2}(1+o(x')) \,,
$$
as $x'\to 0$, and, by letting $y'=c^{\frac{1}{2}}\delta^{-\frac{1}{2}}\mathcal{Q}^{\frac{1}{2}}x'$ we obtain
\begin{equation*}
 \widehat{I}= \left(c\,\,{\rm det}(\mathcal{Q})\right)^{-\frac{N-1}{2}}C_N \left|\partial_{\xi_N} H_0(P_0)\right|^{-1}\Psi_N^{-1}(\delta)(1+o(1))
\end{equation*}
as $\delta \to 0^+$. $\Psi_N(\delta)$ is given by \eqref{PsiNdelta}, where we used Remark \ref{remark calcolointegrale}.

 From \eqref{int_1} and \eqref{int_2} we obtain \eqref{diff_U1U2}. The assumption of the theorem follows from the mean value theorem.

\begin{remark} \label{remark calcolointegrale}
{\em Let $z \in \mathbb{R}^{N-1}$, $I_\delta$ be given by
$$
I_\delta = \int_{|z|<\frac{1}{\sqrt{\delta}}} \frac{dz}{1+|z|^2}
$$
and
$$
\psi_N(\delta) =
\begin{cases}
1 & N=2\,, \\
-\log \delta & N=3 \,, \\
\delta^{-\frac{N-3}{2}}  & N\geq 4 \,,
\end{cases}
$$
Then
$$
\lim_{\delta \to 0} C_N\psi_N(\delta)^{-1} I_\delta =1\,,
$$
where $C_N$ is a constant depending only on the dimension $N$.}
\end{remark}

\appendix

\section{Basic facts for the anisotropic conductivity problem} \label{appendix1}
Let $\Omega$ be a subset of $\mathbb{R}^N$ and $\{D^i\}_{i \in \{1,\ldots,m\}}$ be a family open domains, such that $\overline{D}^i \cap \overline{D}^j = \emptyset$ for $i \neq j$, with boundaries of class $C^{2,\alpha}$, with $0 < \alpha < 1$. Let

\begin{equation*}
  D = \bigcup_{i=1}^mD^i \,.
\end{equation*}

Let $\Omega_{D}=\Omega \setminus \overline D$ and let $\varphi \in C^{2,\alpha}(\Omega)$. As mentioned in the Introducion, the perfectly conductivity problem is the following
\begin{equation} \label{extreme} \tag{$E_H$}
\left\{
   \begin{array}{ll}
    \text{div}\left(H(\nabla u)\nabla_{\xi} H(\nabla u)\right) = 0& \hbox{in} \,\,\, \Omega_{D}, \\
 u_{+} = u_- & \hbox{in} \,\,\, \partial D \,, \\
H(\nabla u)=0 & \hbox{in} \,\,\, D \,, \\
     \displaystyle\int_{\partial D^i} H\left(\nabla u\right)\nabla_{\xi}H\left(\nabla u\right)\cdot \nu ds=0 & i=1,\ldots, m,\\
       u=\varphi & \hbox{on} \,\,\, \partial \Omega \,,
    \end{array}
\right.
\end{equation}
where $\nu$ denotes the outward unit normal to $D$ and $\Omega$.

 By regularity elliptic theory we have that $u \in C^{1,\alpha}(\Omega_D)$ (see \cite{Dibenedetto}) and $H(\nabla u)\nabla_{\xi} H(\nabla u) \in W^{1,2}_{\text{loc}}(\Omega)$ (see \cite{AvelinKuusiMingione,CianchiMazya}).

\begin{theorem}
  There exists at most one solution $u \in H^1(\Omega_D) \cap C^{1,\alpha}(\overline{\Omega}_{D})$ of problem \eqref{extreme}.
\end{theorem}

\begin{proof}
  Let  $u_1$, $u_2 \in H^1(\Omega_D)$ be two solutions of \eqref{extreme}. By multiplying the first equation of \eqref{extreme} by $u_1-u_2$ and integrating by parts, for $j \in \{1,2\}$, we have

\begin{eqnarray*}
  0 &=& \displaystyle\int_{\Omega_{D}} H\left(\nabla u_j\right)\nabla_{\xi}H\left(\nabla u_j\right)\cdot \nabla (u_1-u_2) dx - \displaystyle\int_{\partial \Omega} H(\nabla u_j)\nabla_{\xi} H(\nabla u_j) (u_1-u_2) \cdot \nu ds \nonumber \\
& & \quad +  \sum_{i=1}^m \displaystyle\int_{\partial D^i} H(\nabla u_j)\nabla_{\xi} H(\nabla u_j) (u_1-u_2) \cdot \nu ds\nonumber  \\
&=&  \displaystyle\int_{\Omega_{D}} H\left(\nabla u_j\right)\nabla_{\xi}H\left(\nabla u_j\right)\cdot \nabla (u_1-u_2) dx  \,,\nonumber  \\
\end{eqnarray*}
where in the last equality we used the fourth condition in \eqref{extreme} and the fact that $u_1=u_2$ on $\partial \Omega$.
Thus, by the strong convexity of $H$, we have
\begin{equation*}
  0=\displaystyle\int_{\Omega_{D}} \left( H\left(\nabla u_1\right)\nabla_{\xi}H\left(\nabla u_1\right)-H\left(\nabla u_2\right)\nabla_{\xi}H\left(\nabla u_2\right)\right)\cdot \nabla (u_1-u_2) dx \geq \lambda \displaystyle\int_{\Omega_{D}} \left|\nabla (u_1-u_2)\right|^2 dx \geq 0.
\end{equation*}
Thus $\nabla u_1= \nabla u_2$ in $\Omega_{\delta}$ and, since $u_1=u_2$ on $\partial D^i$, we have $u_1=u_2$ in $\Omega_{D}$.
\end{proof}

We define the energy functional
\begin{equation*}
  I_{\infty}[u]=\dfrac{1}{2} \displaystyle\int_{\Omega_{\delta}} H\left(\nabla u\right)^2 dx,
\end{equation*}
where $u$ belongs to the set
\begin{equation*}
  \mathcal{A} := \left\{u \in W^{1,2}_{\varphi}(\Omega) : H\left(\nabla u\right) =0 \,\, \text{on} \,\, \overline{D}\right\}.
\end{equation*}

\begin{theorem}
There exists a minimizer $u \in\mathcal{A}$ satisfying
\begin{equation*}
   I_{\infty}[u]=\min_{v \in \mathcal{A}} I_{\infty}[v].
\end{equation*}
 Moreover, $u \in W^{1,2}(\Omega_D) \cap C^{1,\alpha}(\overline{\Omega}_{D})$ is a solution to \eqref{extreme}.
\end{theorem}

\begin{proof}
The existence of the minimizer and the Euler Lagrange equation $\Delta_H u = 0$ follows from standard methods in the calculus of variations.
 The only thing which we need to shown is the fourth equation of \eqref{extreme}. Let $i \in \{1,\ldots,m\}$ be fixed and let $\phi \in C^{\infty}_0(\Omega)$ be such that
\begin{equation*}
  \phi = \left\{
              \begin{array}{lllll}
                1, & \hbox{on}& \partial D^i, & &\\
                0, & \hbox{on}& \partial D^j, & \hbox{for}& j\neq i.
              \end{array}
            \right.
\end{equation*}
Since $u$ is a minimizer, by integrating by parts we obtain
\begin{eqnarray*}
  0 &=& - \displaystyle\int_{\Omega_{D}} \text{div}\left(H(\nabla u)\nabla_{\xi} H(\nabla u)\right) \phi \, dx \nonumber \\
 &=& \displaystyle\int_{\Omega_{D}} H\left(\nabla u\right)\nabla_{\xi}H\left(\nabla u\right)\cdot\nabla \phi \, dx - \displaystyle\int_{\partial \Omega} H(\nabla u)\nabla_{\xi} H(\nabla u) \phi \cdot \nu ds \nonumber \\
&+ & \sum_{j=1}^m \displaystyle\int_{\partial D^j} H(\nabla u)\nabla_{\xi} H(\nabla u) \phi \cdot \nu ds\nonumber  \\
&=&  \displaystyle\int_{\partial D^i} H\left(\nabla u\right)\nabla_{\xi}H\left(\nabla u\right)\cdot \nabla \phi \,,dx \nonumber  \\
\end{eqnarray*}
and we conclude.
\end{proof}

\section{Estimates for the radii of the touching balls in the proof of Theorem \ref{thm_main_1}}\label{Estimates for the radii of the touching balls}
In this Appendix we prove two technical lemmas needed in the proof of Theorem \ref{thm_main_1}. We recall that $D_\delta^1$ and $D_\delta^2$ are Wulff shapes of radii $R_1$ and $R_2$, respectively.

In the first lemma, for a point $P \in \partial D_\delta^1$ we consider the ball of radius $r_1$ touching $\partial D_\delta^1$ at $P$ from the inside; $r_2$ is the radius of the concentric ball which touches $D_\delta^2$ from the outside (see Fig. \ref{Figurar1r2_1}).
\begin{center}
\begin{figure}[h]
\includegraphics[scale=0.5]{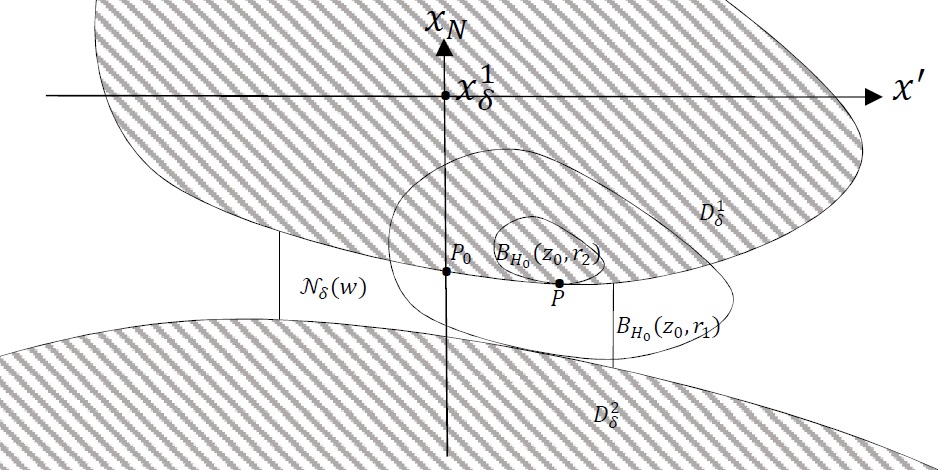}
\caption{\label{Figurar1r2_1}}
\end{figure}
\end{center}
\begin{lemma} \label{lemma_diff_radii}
Let $s \in (0,1]$ and let $P_0,P,r_2$ and $r_1$ be as in the proof of Theorem \ref{thm_main_1}, then
\begin{equation}\label{r2_meno_r1_inlemma}
r_2 - r_1 = \delta + (1-s)\dfrac{R_1+R_2}{2R_1R_2} \mathcal{Q}P^\perp \cdot P^\perp + o(\delta^2 + |\omega|^2)
\end{equation}
as $\delta$ and $|P-P_0|$ go to zero, and where $P^\perp$ is the projection of $P$ on the orthogonal to $P_0$.
\end{lemma}

\begin{proof}
Without loss of generality, we may assume that the ball $D_\delta^1$ has center at the origin and $D_\delta^2$ has center in $Z=(0,\ldots,0, Z_N)$, with $Z_N<0$ and  $H_0(Z)= R_1+R_2+\delta$. Let $Q$ be the center of the ball of radius $r_1=t R_1$, $t \in (0,1]$, touching $\partial D_\delta^1$ at $P$ from the inside, and let $r_2$ be the radius of the ball centered at $Q$ which is tangent to $\partial D_\delta^2$. In particular
\begin{equation} \label{dic1}
r_2+R_2 = H_0(Q-Z) \,.
\end{equation}

It is clear that, denoting by $\nu$ and $\nu_H$ the Euclidean and anisotropic norms, respectively, we have
\begin{equation*}
  \nu_H := \nu^{ext}_H= \nabla_{\xi}H\left(\nu(x)\right) = \nabla_{\xi}H\left(-\dfrac{\nabla u}{|\nabla u|}\right)= \dfrac{x}{H_0(x)}= \dfrac{x}{R_1},
\end{equation*}
at any point on $\partial D_\delta^1$. Being
\begin{equation}\label{P-Q1H}
  Q-P=-tR_1  \nu_{H}(P),
\end{equation}
\begin{equation}\label{PH}
  P=R_1 \nu_H(P),
\end{equation}
 and
\begin{equation}\label{ZH}
  Z=(R_1+R_2 +\delta) \nu_H(P_0),
\end{equation}
then, by using \eqref{P-Q1H}, \eqref{PH} and \eqref{ZH}, we have that \eqref{dic1} can be written as
\begin{eqnarray}\label{r1+RH}
r_2+R_2 & = & H_0(Q-Z)= H_0(Q-P+P-Z)\nonumber \\
&=& H_0\left((1-t) R_1 \nu_H(P)- (R_1+R_2+\delta) \nu_H(P_0)\right)\nonumber \\
&=& H_0\left((1-t) R_1 \nu_H(P_0) + (1-t) R_1 (\nu_H(P) - \nu_H(P_0)) - (R_1+R_2+\delta) \nu_H(P_0)\right)\nonumber \\
&=& R_1 H_0\left(\left(\frac{R_2}{R_1}+t \right)  \nu_H(P_0) + \psi  \right)\nonumber
 \end{eqnarray}
where
$$
\psi = \frac{\delta}{R_1} \nu_H(P_0) - (1-t)  (\nu_H(P) - \nu_H(P_0))
$$
is small for $\delta$ small and $P$ close to $P_0$.
By Taylor expansion and using the homogeneities properties of $H_0$, we have
\begin{equation*}
\begin{split}
\frac{r_2}{R_2} + 1 &  = \frac{R_1}{R_2}  \left\{\left(\frac{R_2}{R_1}+t \right)  H_0(\nu_H(P_0)) + \nabla H_0(\nu_H(P_0)) \cdot \psi + \frac12 (1+t)^{-1} \nabla^2 H_0(\nu_H(P_0)) \psi \cdot \psi  + o(|\psi|^2)\right\}  \\
& = \frac{R_1}{R_2}  \left\{  \frac{R_2}{R_1} +t + \nabla H_0(\nu_H(P_0)) \cdot \psi  + \frac12 (1+t)^{-1} \nabla^2 H_0(\nu_H(P_0)) \psi \cdot \psi   + o(|\psi|^2) \right\}\,.
\end{split}
\end{equation*}
as $\delta \to 0^+$ and $P \to P_0$. Since $tR_1=r_1$ we have
\begin{equation}\label{r_2-r_1_Istep}
\frac{r_2 - r_1}{R_2} =  \dfrac{R_1}{R_2}\left(\nabla H_0(\nu_H(P_0)) \cdot \psi + \frac12  (1+t)^{-1}\nabla^2 H_0(\nu_H(P_0)) \psi \cdot \psi + o(|\psi|^2) \right) \,.
\end{equation}
Since
$$
\nabla H_0\left(\nu_H(P_0)\right)\nu_H(P_0) = 1 \,,
$$
then
$$
\nabla H_0(\nu_H(P_0)) \cdot \psi  = \frac{\delta}{R_1} - (1-t) \nabla H_0(\nu_H(P_0)) \cdot (\nu_H(P) - \nu_H(P_0)) \,.
$$
and being
$$
\nabla^2 H_0\left(\nu_H(P_0)\right)  \nu_H(P_0)  = 0 \,,
$$
we find
$$
 \nabla^2 H_0(\nu_H(P_0)) \psi \cdot \psi  =  (1-t)^2 \nabla^2 H_0(\nu_H(P_0)) (\nu_H(P) -\nu_H(P_0)) \cdot (\nu_H(P) -\nu_H(P_0)) \,.
$$
From \eqref{r_2-r_1_Istep} we obtain
\begin{equation}\label{r_2-r_1_IIstep}
\frac{r_2 - r_1}{R_2} =  \frac{\delta}{R_2} +\dfrac{R_1}{R_2}\left[ -(1-t) \nabla H_0(\nu_H(P_0)) \cdot \omega + \frac{(1-t)^2}{2\left(\dfrac{R_2}{R_1}+t\right)} \nabla^2 H_0(\nu_H(P_0)) \omega \cdot \omega + o(\delta^2 + |\omega|^2)  \right]\,.
\end{equation}
where we set
$$
\omega = \nu_H(P) -\nu_H(P_0) \,.
$$

Now we observe that $1 = H_0\left(\nu_H(P)\right)= H_0\left(\nu_H(P_0)+\omega\right) $ which gives
$$
  1=  H_0\left(\nu_H(P_0)\right) +\nabla H_0\left(\nu_H(P_0)\right) \cdot \omega +
\dfrac{1}{2}  \nabla^2 H_0\left(\nu_H(P_0)\right)\omega \cdot \omega + o(|\omega|^2) \,,
$$
so that, being $H_0(\nu_H(P_0))=1$,
\begin{equation*}
  -  \nabla H_0\left(\nu_H(P_0)\right) \cdot \omega =\dfrac{1}{2}
 \nabla^2 H_0\left(\nu_H(P_0)\right)\omega \cdot \omega+  o(|\omega|^2),
\end{equation*}
and \eqref{r_2-r_1_IIstep}
\begin{equation*}\label{r_2-r_1_IIstep}
r_2 - r_1 = \delta + R_1 \frac{1-t}{2}\dfrac{1+\dfrac{R_2}{R_1}}{t+\dfrac{R_2}{R_1}} \nabla^2 H_0(\nu_H(P_0)) \omega \cdot \omega + o(\delta^2 + |\omega|^2) \,.
\end{equation*}
From \eqref{H=0}, we notice that the range of $\nabla^2 H_0(\xi)$ lies in $\xi^\perp$ and hence
\begin{equation}\label{r_2-r_1_IIstep}
r_2 - r_1 = \delta +  R_1 \frac{1-t}{2}\dfrac{1+\dfrac{R_2}{R_1}}{t+\dfrac{R_2}{R_1}} \nabla^2 H_0(\nu_H(P_0)) \nu_H(P)^\perp \cdot \nu_H(P)^\perp + o(\delta^2 + |\omega|^2) \,,
\end{equation}
where $\nu_H(P)^\perp$ is the projection of $\nu_H(P)$ on the orthogonal to $\nu_H(P_0)^\perp$. Since $P$ and $P_0$ are on the boundary of the Wulff shape, we have $\nu_H(P)=P/R_1$, $\nu_H(P_0)=P_0/R_1=\hat{P}$ and from \eqref{matriceQ} we obtain \eqref{r2_meno_r1_inlemma}.
\end{proof}

In the following lemma, for a point $P \in \partial D_\delta^1$ we consider a ball of radius $\rho_2$ touching $\partial D_\delta^1$ at $P$ from the outside and having center inside $D_\delta^2$; $\rho_1$ is the radius of the concentric ball which touches $D_\delta^2$ from the inside (see Fig. \ref{Figurarho1rho2_2}).

\begin{lemma} \label{lemma_diff_radii_2}
Let $P_0,P,\rho_2$ and $\rho_1$ be as in the proof of Theorem \ref{thm_main_1}. There exists a constant $C$ independent of $\delta$ and $w$ such that for any $\delta + C|w|<t<\dfrac{1}{2}$, with $\delta$ and $|w|$ sufficiently small, we have
\begin{equation}\label{r2_meno_r1_inlemma}
\rho_2 - \rho_1 = \delta + (1+t) \dfrac{R_1+R_2}{2R_1R_2} \mathcal{Q} P^\perp \cdot P^\perp + o(\delta^2 + |P-P_0|^2)
\end{equation}
as $\delta$ and $|P-P_0|$ go to zero, and where $P^\perp$ is the projection of $P$ on the orthogonal to $P_0$.
\end{lemma}

\begin{proof}
By arguing as in the proof of Lemma \ref{lemma_diff_radii} we have that $\rho_1$ and $\rho_2$ are related by the following identity
$$
R_2 - \rho_1 = H_0((R_1+\rho_2)\nu_H(P) - (R_1+R_2+\delta) \nu_H(P_0)) \,.
$$
By simple manipulations we have
$$
R_2-\rho_1 = (R_2-\rho_2) H_0( \nu_H(P_0) + \phi ),
$$
where
$$
\phi = \frac{(R_1+ \rho_2)(\nu_H(P) - \nu_H(P_0)) - \delta \nu_H(P_0)}{\rho_2-R_2} \,.
$$
Hence
$$
R_2-\rho_1 = (R_2-\rho_2) \left\{ H_0( \nu_H(P_0)) + \nabla H_0( \nu_H(P_0)) \cdot \phi + \frac12 \nabla^2 H_0( \nu_H(P_0)) \phi \cdot \phi  + o(|\phi|^2) \right\},
$$
as $\delta \to 0$ and $P\to P_0$, and being $H_0(\nu_H(P_0))=1$, we find
$$
\rho_2-\rho_1 = (R_2-\rho_2) \left\{  \nabla H_0( \nu_H(P_0)) \cdot \phi + \frac 12 \nabla^2 H_0( \nu_H(P_0)) \phi \cdot \phi  + o(|\phi|^2) \right\} \,.
$$
As done in the previous lemma, we have that
\begin{equation*}
\begin{split}
 \nabla H_0( \nu_H(P_0)) \cdot \phi & =  -  \frac{R_1+ \rho_2}{R_2 - \rho_2} \nabla H_0( \nu_H(P_0)) \cdot   \omega  + \frac{\delta}{R_2-\rho_2} \nabla H_0( \nu_H(P_0)) \cdot\nu_H(P_0)\\
& =  \frac{\delta}{R_2-\rho_2}  -  \frac{R_1+ \rho_2}{R_2 - \rho_2} \nabla H_0( \nu_H(P_0)) \cdot   \omega   \\
& =  \frac{\delta}{R_2-\rho_2} + \frac{R_1+ \rho_2}{R_2 - \rho_2} \dfrac{1}{2}
 \nabla^2 H_0\left(\nu_H(P_0)\right)\omega \cdot \omega+  o(|\omega|^2),
\end{split}
\end{equation*}
and we find
$$
\rho_2 - \rho_1 = \delta + \frac{R_1+\rho_2}{2}\dfrac{R_1+R_2}{R_2-\rho_2} \nabla^2 H_0\left(\nu_H(P_0)\right) \nu_H(P) \cdot \nu_H(P)+  o(|\omega|^2 + \delta^2) \,.
$$
We choose $\rho_2= \overline{t} R_1$ with $\overline{t}>0$. Since $\partial D_{\delta}^1$ and $\partial D_{\delta}^2$ are smooth, there exists a constant $\overline{C}>0$ such that $\overline{t}>\delta + \overline{C}|w|^2$. Being $\nu_H(P)=P/R_1$ and $\nu_H(P_0)=P_0/R_1$, we conclude.
\end{proof}

\begin{center}
\begin{figure}[h]
\includegraphics[scale=0.5]{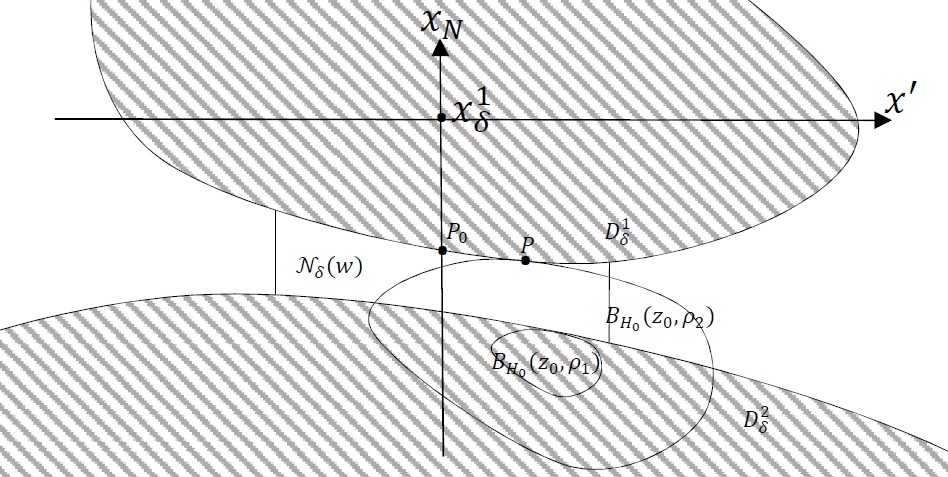}
\caption{\label{Figurarho1rho2_2}}
\end{figure}
\end{center}

{\renewcommand{\addtocontents}[2]{}
\section*{Acknowledgements}}

This work was supported by the project FOE 2014 ``Strategic Initiatives for the Environment and Security - SIES" of the Istituto Nazionale di Alta Matematica (INdAM) of Italy.
The authors have been partially supported by  the ``Gruppo Nazionale per l'Analisi Matematica, la Probabilit\`{a} e le loro Applicazioni (GNAMPA)'' of the ``Istituto Nazionale di Alta Matematica'' (INdAM).

\end{document}